\documentclass[11pt]{article}

\usepackage{amsmath}
\usepackage{amsfonts}
\usepackage{amssymb}
\usepackage{amsmath}
\usepackage{amsthm}
\usepackage{latexsym}
\usepackage{graphicx,cite}
\usepackage{float}
\usepackage{a4wide}

\usepackage{mathtools}
\mathtoolsset{showonlyrefs}

\usepackage{xcolor}
\definecolor{hanblue}{rgb}{0.27, 0.42, 0.81}
\definecolor{red}{rgb}{1.0, 0.0, 0.0}
\usepackage[colorlinks, citecolor=hanblue,linkcolor=red]{hyperref}

\newtheorem{thm}{Theorem}[section]
\newtheorem{lem}[thm]{Lemma}
\newtheorem{prop}[thm]{Proposition}

\newtheorem{problem}[thm]{Problem}

\theoremstyle{definition}
\newtheorem{defn}[thm]{Definition}

\newtheorem{prob}[thm]{Problem}

\theoremstyle{remark}
\newtheorem{rem}[thm]{Remark}

\numberwithin{equation}{section}

\usepackage{tikz,fp,ifthen,fullpage}
\usepackage{pgfmath}
\usetikzlibrary{backgrounds}
\usetikzlibrary{decorations.pathmorphing,backgrounds,fit,calc,through}
\usetikzlibrary{arrows}
\usetikzlibrary{shapes,decorations,shadows}
\usetikzlibrary{fadings}
\usetikzlibrary{patterns}
\usetikzlibrary{mindmap}
\usetikzlibrary{decorations.text}
\usetikzlibrary{decorations.shapes}

\title{Some minimization problems\\ for planar networks of elastic curves}
\author{Anna Dall'Acqua 
\footnote{Institut f\"{u}̈r Analysis, Fakult\"{a}t f\"{u}r Mathematik und Wirtschaftswissenschaften, Universit\"{a}̈t Ulm, 89081 Ulm, Germany}
\and Alessandra Pluda
\footnote{Fakult\"at f\"ur Mathematik, Universit\"at Regensburg, Universit\"atsstrasse 31, 
93053 Regensburg, Germany}
}

\begin{document}

\maketitle

\begin{abstract}
In this note we announce some results that will appear in~\cite{danovplu}
on the minimization of the functional 
$F(\Gamma)=\int_\Gamma k^2+1\,\mathrm{d}s$,
where $\Gamma$ is a network of three curves with fixed equal angles at the two junctions.
The informal description of the results is 
accompanied by a partial review of the theory of elasticae
and a diffuse discussion about the onset of interesting variants of the original problem
passing from curves to networks.
The considered energy functional $F$ is given by 
the elastic energy and a term that
penalize the total length of the network. We will show that penalizing the length is
tantamount to fix it.
The paper is concluded with the explicit computation of the penalized elastic energy
of the  \lq\lq Figure Eight\rq\rq, 
namely the unique closed elastica with self--intersections (see Figure~\ref{otto}). 
\end{abstract}



\bigskip


\section{Introduction}

A $N$--network is the union of a
finite number $N$ of sufficiently smooth planar curves $\gamma^i$ 
whose end points meet in junctions. 
We consider the elastic energy functional for a $N$--network  
$\Gamma$ with curves  $\gamma^i \in H^2$ defined as:
\begin{equation}\label{funzionale}
E(\Gamma):=\int_{\Gamma} k^2\,{\rm{d}}s=\sum_{i=1}^N\int_{\gamma^i}\left(k^i\right)^2\,{\rm{d}}s^i\,,
\end{equation}
where $k^i$ and $s^i$ denote respectively the scalar curvature and the arclength parameter of $\gamma^i$.

We are interested in problems that involve the minimization of 
the energy $E$ on networks.

Before studying the case of networks,  we first recall some known results in the case of curves.
In particular, we will see in Section~\ref{curves} 
that the minimization of the elastic energy on closed regular curves is not a well posed problem. 
A possible solution, listed between 
others, 
is then to penalize or to fix the length of the curve. 
In the first case we speak of the \emph{penalized} problem 
and in the second of the \emph{constrained} problem.
It is known that 
these two problems are equivalent (see Section~\ref{secresc}).
In the case of networks an ill posedness issue appears also if we consider
the penalized or the constrained problem. 
At the beginning of Section~\ref{sezionenetwork} 
thanks to a comparison with the
case of curve
we give a possible condition that makes the problem
 well posed also when we consider  networks.
 
Then we describe in some details the 
Theta--networks and the relatives results proved in~\cite{danovplu}.
 
Section~\ref{secEL} is devoted to the computation of the Euler-Lagrange equations.
These computations are directly applicable to networks in $\mathbb{R}^n$. 
Also the variational conditions at the junctions are presented.

In Section~\ref{secresc} the behaviour under rescaling
of the penalized functional is studied and it is shown that 
in fact the penalized and the constrained problem are equivalent. 
This is done actually in quite a general framework that might apply 
to  other geometric minimization problems.

Finally in the last section it is contained the new result of this note.
Using the  
representation given in \cite{explicitelasticae}
of the critical points of the penalized elastic energy, 
we compute the energy of the 
\lq\lq Figure Eight\rq\rq~(see Figure~\ref{otto}), 
namely the unique closed elastica (critical point of the elastic energy with a length constraint) with self--intersections. 
This computation is crucial for the arguments \cite{danovplu} as we will explain below.

We conclude this introduction pointing out to the reader that 
in Sections~\ref{curves} and~\ref{sezionenetwork} we state also some problems
that to the authors' knowledge 
have
not been studied in the literature yet.

\section{The case of curves}\label{curves}

Let us now start by fixing the precise notation that we will use in the whole paper.

When we consider a curve $\gamma$, we mean
a parametrization  $\gamma:[0,1]\to\mathbb{R}^2$.
A curve is of class $C^k$ (or $H^k$) with $k=1,2,\ldots$
if it admits a parametrization $\gamma$ of class $C^k$ (or $H^k$,  respectively).
We remind that a curve is said to be regular 
if $\vert \partial_x \gamma(x) \vert\neq 0$ for every $x\in[0,1]$.
Denoting with  $s$ the arclength parameter, 
we have  $\partial_s = \frac{1}{|\partial_x \gamma(x)|} \partial_x$. 

Consider a regular smooth curve $\gamma$, 
then the unit tangent vector of $\gamma$ is
$$
{\tau}:=\partial_{s}\gamma=\frac{\partial_x\gamma}{\left|\partial_x\gamma\right|} 
\,.
$$
The unit normal vector $\nu$ to $\gamma$ is defined 
as the counter-clockwise rotation of $\pi/2$ of $\tau$.
The curvature vector of $\gamma$ is defined as 
\begin{equation}\label{curvvector}
\boldsymbol{\kappa}:= \partial_{ss}\gamma
=k\nu=\partial_s \tau
=\frac{\partial_{xx}\gamma}{\left|\partial_x\gamma\right|^{2}}
 -\frac{\left\langle \partial_{xx}\gamma,\partial_{x}\gamma\right\rangle\partial_x\gamma}{\left|\partial_x\gamma\right|^{4}}\,.
 \end{equation}
with $k$ the scalar curvature given by $k=|\partial_{ss}\gamma|$.  
The following relations, direct consequences of the definitions above, will be relevant
in the computation of Section~\ref{secEL}:
\begin{align}\label{dercurvatura}
\partial_s \nu=-k\tau\,, \quad \partial_s\boldsymbol{\kappa}=\partial_sk \ \nu-\left(k\right)^{2}\tau\, 
\mbox{ and } \;\partial_{ss}\boldsymbol{\kappa}
=\partial_{ss} k \ \nu-\left(k\right)^{3}\nu-3k \partial_{s}k \ \tau\,.
\end{align}
Moreover, we will adopt the following convention for integrals,
$$
\int_{\gamma}f(\gamma,\tau,\nu,...)\,\mathrm{d}s
=\int_{0}^1f(\gamma,\tau,\nu,...)\vert \partial_x\gamma\vert\,\mathrm{d}x\,,
$$
as the arclength measure is given by $\mathrm{d}s=\vert \partial_x\gamma\vert\mathrm{d}x$
on every curve $\gamma$.

\medskip

For a regular curve $\gamma$ of class $H^2$ 
the elastic energy is defined as 
\begin{equation}\label{energiacurve}
E(\gamma):=\int_{\gamma}k^2\,{\rm{d}}s\,.
\end{equation}
It is nice to mention that this energy
was considered already by Bernoulli  to 
model elastic rods, moreover the first solution of the associated variational problem was given by Euler, both contributions around 1743 (see~\cite{Truesdell}).

\medskip

Classically an \emph{elastica} is a 
critical point of the functional~\eqref{energiacurve}
defined on regular curves of fixed length satisfying given first order boundary data. 
When the constraint on the length is removed, one usually speaks of free elasticae.
The closed (free) elasticae are classified thanks to the results by Langer and Singer 
in the beautiful paper~\cite{langsing1}.
In particular in~\cite[Theorem 0.1 (a)]{langsing1} they prove that 
the circle and  the  ``\emph{Figure Eight}"  (or a multiple cover of one of these two)
are the unique closed planar elasticae.

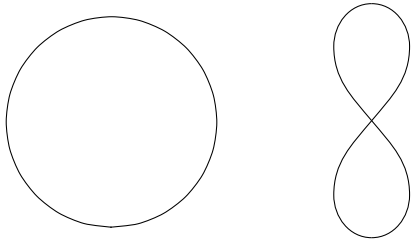
\begin{figure}[H]
\begin{center}
\begin{tikzpicture}[scale=1.4]
\draw[color=black,scale=0.5,domain=-3.14: 3.15,
smooth,variable=\t,shift={(0,0)},rotate=0]plot({2.*sin(\t r)},
{2.*cos(\t r)}) ;
\draw[white]
(-1,-1.1)--(1,-1.1);
 \end{tikzpicture}\qquad\qquad
 \begin{tikzpicture}[scale=0.9]
\draw
(0,1.73)to[out= 0,in=90, looseness=1] (0.56,1.1)
(0.56,1.1)to[out= -90,in=50, looseness=1] (0,0);
\draw
(0,1.73)to[out= 180,in=90, looseness=1] (-0.56,1.1)
(-0.56,1.1)to[out= -90,in=130, looseness=1] (0,0);
\draw[rotate=180]
(0,1.73)to[out= 0,in=90, looseness=1] (0.56,1.1)
(0.56,1.1)to[out= -90,in=50, looseness=1] (0,0);
\draw[rotate=180]
(0,1.73)to[out= 180,in=90, looseness=1] (-0.56,1.1)
(-0.56,1.1)to[out= -90,in=130, looseness=1] (0,0);
\end{tikzpicture}
\end{center}
\caption{The circle and the ``Figure Eight" are the unique
closed planar elasticae.}\label{otto}
\end{figure}

Let us now restrict from stationary points 
to minimizers. It is natural to first consider
the minimization of the elastic energy for one \emph{closed} curve.
We notice that without further conditions this problem is not well posed since
$$
\inf \left\lbrace E(\gamma)\;\big\vert\; \gamma\;\text{is a closed regular curve of
class}\,H^2\right\rbrace =0\,,
$$
and the infimum is not attained. Indeed consider a sequence of circles $\mathcal{C}_R$ 
with radius $R$,
then $E(\mathcal{C}_R)$ goes to zero as $R\to \infty$ and the value zero is not attained,
since the elastic energy has value zero only on (pieces of) straight lines.

This first remark makes it clear that an additional constraint in the definition 
of the problem is needed. 
There are many choices for this extra condition, and hence many variants of the 
problem. We give here some possibilities.

\begin{enumerate}
\item One can decide to penalize or 
to fix the length of the curve. We will see later in Lemma~\ref{probequivalenti} 
that these two problems are in fact equivalent. 
In both cases, it is not difficult to prove existence of minimizers by
the direct method in the calculus of variation using that the functional
is lower semicontinuous in $H^2$. In the case of the penalized problem
described by the functional $F(\gamma):=E(\gamma)+L(\gamma)$ it is useful
to consider a minimizing sequence parametrized with constant speed equal to the length. Then
to achieve compactness only one extra argument is needed:
a bound from below on the length.
In the case of a simple closed curve Gauss-Bonnet theorem
gives that $\int_\gamma k\,{\rm{d}}s=2\pi$. It is possible to generalise this result in the case of
curves with self--intersections obtaining $\int_\gamma \vert k\vert\,{\rm{d}}s\geq 2\pi$ 
(see~\cite[Appendix A]{danovplu}).
Combing this estimate with H\"older inequality we get
\begin{equation}\label{a}
2\pi\leq \int_{\gamma}\vert k\vert\,{\rm{d}}s\leq
\left(\int_{\gamma}k^2\,{\rm{d}}s\right)^{1/2}\left(\int_\gamma 1\,{\rm{d}}s\right)^{1/2}\,.
\end{equation}
and as a consequence
\begin{equation}\label{b}
F(\gamma)\geq E(\gamma)\geq \frac{4\pi^2}{L(\gamma)}\, ,
\end{equation}
which yields a lower bound on the length. It follows from the arguments in \cite{langsing,langsing1} 
that the unique (up to isometries of $\mathbb{R}^2$) minimizer to the penalized functional
$F(\gamma)$ between all closed regular $H^2$
curves is the circle of radius $1$ and the minimum of the energy is $4\pi \approx 12.56637$. 
In Section \ref{secequiv} 
we will see also that the unique minimizer of the more general problem in which one replace
$F$ by the functional
$F_{\delta}(\gamma):=E(\gamma)+\delta L(\gamma)$ with $\delta >0$, 
is an appropriate rescaling (depending on $\delta$) of the unit circle.

In~\cite{danovplu} we minimize the 
functional $F$ on drops that are defined as follows.

\begin{defn}\label{defdrop}
We say that a regular curve $\gamma:[0,1]\to\mathbb{R}^2$ 
is a \emph{drop} if $\gamma(0)=\gamma(1)$. 
\end{defn}
\begin{prob}\label{problemdrops}
Does a minimizer to 
$\inf\{F(\gamma)\,\vert \;\gamma\, \text{ is a drop of class }H^2\}$
exist?
\end{prob}
In~\cite{danovplu} we 
give a positive answer to this question, moreover in~\cite[Proposition 3.12]{danovplu}
we prove that one of the two drops of a proper rescaling of the 
the ``Figure Eight" (see Figure~\ref{otto}) is the unique minimizer 
(up to isometries of $\mathbb{R}^2$) 
for Problem~\ref{problemdrops}. 
Following the approach introduced in~\cite{explicitelasticae}
in Section~\ref{bulgari} of this work we show that the energy of the minimizer 
is $\approx10.60375$.

We want to underline the difference between drops and closed curves:
if a curve $\gamma$ of class $H^2$ is closed then $\gamma$ has a $1$-periodic $C^1$ 
extension to $\mathbb{R}$. 
Whereas for $H^2$ drops
at the point $\gamma(0)=\gamma(1)$ the derivatives do not have to match. 
We allow for instance the possibility of having an angle 
or a cusp, hence we are requiring less regularity on drops than on closed curves.
Working in the larger class of drops, 
the energy of the minimizer (not surprisingly) decreases.

\item Another option is requiring that the area enclosed 
by the simple closed curve is bounded by some fixed constant.  
To be more precise consider a smooth,
simply connected and bounded open set $\Omega$,
bounded by a 
Jordan curve $\gamma$. Then $E(\partial\Omega)$
is nothing else than $E(\gamma)$ with $E$ defined as in~\eqref{energiacurve}.
Calling $A(\Omega)$ the area of $\Omega$ and given $A_0>0$,
one looks for
$$
\min \left\lbrace E(\partial\Omega) \big\vert\; A(\Omega)\leq A_0\right\rbrace\,,
$$
or equivalently (see~\cite{bucurhenrot} or Lemma~\ref{probequivalenti} )
$$
\min  \left\lbrace E(\partial\Omega)+A(\Omega)\right\rbrace\,.
$$
In~\cite{bucurhenrot, ferkawni} it is proved that the unique minimizer is the disc 
and it is also shown the new isoperimetric inequality 
$$
E^2(\partial\Omega)A(\Omega)\geq 4 \pi^3\,.
$$
The existence result in this case is more difficult than in the case of length penalization.
The strategy of proof due to Bucur and Henrot has also lead the authors
to study the case in which the curve 
that bounds $\Omega$ is not globally $C^1$, but it has one cusp. 
Precisely, they minimize the elastic energy on \lq\lq closed loops without self-intersections points, 
which are smooth except at one point, where the tangents are opposite\rq\rq, 
cited from~\cite{bucurhenrot}.
They call these curves also drops. 
The main difference to our definition is that we allow for any angle 
between the tangent vectors at the junction. 
As we just wrote above, in the case of the elastic energy 
with length penalisation considered in~\cite{danovplu} the optimal drop is given 
by the (upper) half of the \lq\lq Figure Eight\rq\rq (see Figure~\ref{otto}) 
and the angle between the tangents at the junction 
is $\approx 80$ degrees (see~\cite{explicitelasticae}).
It is worth to mention 
that in~\cite{bucurhenrot} the authors cannot use the
result by~\cite{langsing1} on the classification of the elasticae because of the
area constraint, what we instead heavily use in~\cite{danovplu}.

\item A natural possibility is also not to allow the curves to stay in the whole plane, but 
asking the curves to be \emph{confined} in a open and bounded subset $\Omega$, as done 
in~\cite{domuro,masnou}. 
In~\cite{masnou} existence is shown by the direct method in the calculus of variation 
in the class of confined curves with possibly tangential self-intersections. 
In this situation the most challenging part is the study of regularity 
and qualitative properties of the minimizers. 
In~\cite{masnou} it is proved that a minimal curve is globally $W^{3,\infty}$, 
$\partial_s k$ is a function of bounded variation and 
that each minimizer is smooth away from 
self-intersections and 
contact points. Moreover the authors are able to show that every self-intersection point of an optimal curve has multiplicity not greater than two and that if $\Omega$ is convex, then every optimal curve surrounds a convex set.
\end{enumerate}

Consider now an ``open'' curve $\gamma:[0,1]\to\mathbb{R}^2$
with end points $\gamma(0)\neq\gamma(1)$.
Also in this case the infimum of the energy $E$ between all 
``open'' regular curves of class $H^2$ is zero. In this setting, it is trivially attained  
at every straight segment, hence it is actually a minimum.
To make the problem more interesting one can impose at the boundary
some conditions that force all the competitors not to be segments.
As the elastic functional involves the curvature (in arclength the 
second derivative of the parametrization)
two conditions at both endpoints are expected. The most common choice is the following. 
Fix the points $\gamma(0)=P_0$, $\gamma(1)=P_1$ and the
unit tangent vectors $\tau(0)=\tau_0$, $\tau(1)=\tau_1$ (with $\tau_0$ and $\tau_1$ that do not both have the direction of the segment joining 
$P_0$ and $P_1$). These boundary conditions are usually called \textit{clamped} or \textit{Dirichlet} boundary conditions. 
If $\tau_0$ and $\tau_1$ are in the opposite half--plane with respect to the 
line that pass through the points $P_0$ and $P_1$ we can call this problem
the \textit{minimal elastic lens problem} (see Figure~\ref{lens}).
There are also other possible choices, 
motivated by the study of boundary value problems, 
see for instance~\cite{DG1,DG2,linner,linner2,mandel,miuranew}. 

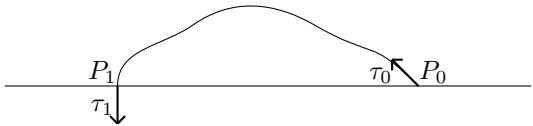
\begin{figure}[h]
\begin{center}
\begin{tikzpicture}
 \draw[thick, shift={(1.65,0.35)}, rotate=45]
(0,0)to[out= -45,in=135, looseness=1] (0.1,-0.1)
(0,0)to[out= -135,in=45, looseness=1] (-0.1,-0.1);
\draw[thick, rotate=-180, shift={(2,0.5)}]
(0,0)to[out= -45,in=135, looseness=1] (0.1,-0.1)
(0,0)to[out= -135,in=45, looseness=1] (-0.1,-0.1);
\draw[thick]
(2,0)--(1.65,0.35)
(-2,0)--(-2,-0.5);
\draw
(-3.5,0)--(3.5,0);
\draw 
(-2,0)
to[out=90, in=-145, looseness=1] (-1, 0.81)
to[out=35,in=150,  looseness=1] (0.62,0.81)
to[out=-30,in=135, looseness=1.5] (2,0);
\path[font=\small]
(2.2,-0.1)[above] node{$P_0$}
(-2.2,-0.1)[above] node{$P_1$}
(-2.2,-0.05)[below] node{$\tau_1$}
(1.5,-0.1)[above] node{$\tau_0$};
\end{tikzpicture}
\caption{A competitor for the minimal ``elastic lens'' problem given as data $P_0,P_1,\tau_0,\tau_1$}\label{lens}
\end{center}
\end{figure}

Recently also the associated 
obstacle problem has been studied, see~\cite{dade} and~\cite{Miura}. 
This problem is similar to the one of confined elastica discussed above, but less geometrical. 
Also in this case minimizers are globally $W^{3,\infty}$ and concave. 

\smallskip

Another interesting problem would be to minimize the elastic energy under a fixed isoperimetric ratio, that is among curves with fixed quotient between length and enclosed area.\footnote{This problem was proposed to the authors by Giulio Giusteri and Luca Lussardi.} It seems that this problem has not  been studied in this situation yet, whereas the corresponding two-dimensional problem for the Willmore energy has been considered for instance in~\cite{schygulla}.

\smallskip

To conclude this short summary about elastic curves we want to underline the fact that 
the literature on the topic is much wider than what we present, for instance we 
refer the reader also to~\cite{bellmu, Sachkov}.

\smallskip

Moreover for completeness, let us say that there is really an extensive literature 
on the elastic flow (the $L^2$-gradient flow of the elastic energy) both for closed and open curves. 
Since it is difficult to give just few references and this is not the subject of this paper, 
we refrain from doing it here.

\section{Networks}\label{sezionenetwork}

We have decided to study the problem for networks because we expect the onset of
new phenomena with respect to the case of the curves. Let us first define precisely
what we mean by network.

\begin{defn}\label{network}
A \textit{$N$--network} is a connected set in the plane, 
union of $N$ curves $\gamma^i:[0,1]\to\mathbb{R}^2$,
whose end points meet in junctions.  
Moreover we say that a $N$--network is \textit{regular} if each curve $\gamma^i$ is regular,
and it is \textit{of class $H^2$} if each curve of the network is of class $H^2$.
\end{defn}

Motivated by what we learnt on curves, one would expect that by adding a term 
that penalizes the length to the elastic energy and, 
for instance, minimizing on the set of networks composed 
by  three regular curves attached in two junctions, we should get a well-posed problem.
Instead, the problem
$$
\inf \{F(\Gamma)\,\vert \, \Gamma \mbox{ regular $3$--network}\} \mbox{ with }F(\Gamma):=E(\Gamma)+L(\Gamma)\,,
$$
and $L(\Gamma)$ the total length of the network, is not well posed. 
Indeed, the infimum is zero and is clearly not attained in the class of regular networks.
A trivial minimizing sequence is given by three equal segments whose length goes to zero. 
We can also give an example of a minimizing sequence with curves 
that meet only at the junctions:
consider a sequence of networks $\Gamma_\varepsilon$ 
composed by two arcs of circle of radius $1$  and of length $\varepsilon$
that meet with a segment (of length $\sqrt{2}\sqrt{1-\cos\varepsilon}\sim\varepsilon$)
with angles of amplitude $\frac\varepsilon2$. 
The energy of $\Gamma_\varepsilon$ is
$F(\Gamma_\varepsilon)=4\varepsilon+\sqrt{2}\sqrt{1-\cos\varepsilon}$ and it converges to zero
when $\varepsilon\to 0$. 

In~\eqref{b} we have seen  that in 
the case of the same functional minimized among closed regular
curves $\gamma$ of class $H^2$ 
there is a lower bound on the energy.
By doing a computation similar to the one in \eqref{a} 
we see now which conditions on the network could give rise to a lower bound on the energy 
and possibly to a well-posed problem on networks. 
The main tool is again the Gauss--Bonnet Theorem and 
for ease of presentation we consider here only 
embedded $3$-networks, i.e. $3$-networks without intersections (except at
the junctions) and without self-intersections.
As in the case of closed curves if the network is not embedded a more refine version of
Gauss--Bonnet is needed, see~\cite{danovplu}.
 An embedded $3$--network can be seen as
two closed curve (with a common piece) 
each having two external angles $\theta_1,\theta_2$ with  
$\theta_i\in [-\pi,\pi]$. Then for each closed curve that compose the network
 the estimate that comes from Gauss--Bonnet becomes
\begin{equation}\label{c}
\int_{\gamma}\vert k\vert\,{\rm{d}}s\geq 2\pi-\theta_1-\theta_2\geq 0\,.
\end{equation}
This shows that by fixing the angles $\theta_i\neq \pi$ we get an uniform bound 
from below for the total curvature, and consequently to
the energy, as in the case of one closed smooth curve.

\subsection{Penalized length and fixed angles}
The previous discussion suggests that
a natural choice for the 
problem is minimizing the elastic energy~\eqref{funzionale}
on the class of regular $H^2$
$N$--networks $\Gamma$,
penalizing the global length of the network and fixing the angles at the junctions. 
Then the considered energy functional is:
\begin{equation}\label{funzionalenetworks}
F(\Gamma):=
\sum_{i=1}^N\int_{\gamma^i}\left(k^i\right)^2\,{\rm{d}}s^i + \sum_{i=1}^N\int_{\gamma^i}1\,{\rm{d}}s^i \,.
\end{equation}

\begin{figure}[h]
\begin{center}
\begin{tikzpicture}[scale=1.3]
\draw
(0,0)to[out= 0,in=-45, looseness=1](1,1)
(1,1)to[out= 135,in=90, looseness=1](0,0);
\draw[rotate=90]
(0,0)to[out= 0,in=-45, looseness=1](1,1)
(1,1)to[out= 135,in=90, looseness=1](0,0);
\draw[rotate=180]
(0,0)to[out= 0,in=-45, looseness=1](1,1)
(1,1)to[out= 135,in=90, looseness=1](0,0);
\draw[rotate=270]
(0,0)to[out= 0,in=-45, looseness=1](1,1)
(1,1)to[out= 135,in=90, looseness=1](0,0);
\end{tikzpicture}\qquad
\begin{tikzpicture}[scale=1.2]
\draw[rotate=36, scale=0.61]
(0.76,0)to[out= 120,in=-48, looseness=1](0.23,0.72)
to[out= -162,in=24, looseness=1](-0.61,0.44)
to[out= -96,in=96, looseness=1](-0.61,-0.44)
(0.76,0)to[out= -120,in=48, looseness=1](0.23,-0.72)
to[out= 162,in=-24, looseness=1](-0.61,-0.44);
\draw
(-1.2,-0.87)to[out= -54,in=-168, looseness=1](-0.37,-1.15)
(-1.22,0)to[out= -120,in=126, looseness=1](-1.2,-0.87);
\draw[rotate=72]
(-1.2,-0.87)to[out= -54,in=-168, looseness=1](-0.37,-1.15)
(-1.22,0)to[out= -120,in=126, looseness=1](-1.2,-0.87);
\draw[rotate=-72]
(-1.2,-0.87)to[out= -54,in=-168, looseness=1](-0.37,-1.15)
(-1.22,0)to[out= -120,in=126, looseness=1](-1.2,-0.87);
\draw[rotate=144]
(-1.2,-0.87)to[out= -54,in=-168, looseness=1](-0.37,-1.15)
(-1.22,0)to[out= -120,in=126, looseness=1](-1.2,-0.87);
\draw[rotate=-144]
(-1.2,-0.87)to[out= -54,in=-168, looseness=1](-0.37,-1.15)
(-1.22,0)to[out= -120,in=126, looseness=1](-1.2,-0.87);
\draw[ rotate=72]
(-1.22,0)to[out= 0,in=180, looseness=1](-0.45,0);
\draw[rotate=-72]
(-1.22,0)to[out= 0,in=180, looseness=1](-0.45,0);
\draw[rotate=144]
(-1.22,0)to[out= 0,in=180, looseness=1](-0.45,0);
\draw[rotate=-144]
(-1.22,0)to[out= 0,in=180, looseness=1](-0.45,0);
\draw
(-1.22,0)to[out= 0,in=180, looseness=1](-0.45,0);
\end{tikzpicture}
\caption{A $4$--network with all angles 
of $90$ degrees and a
$15$--network with all angle 
of $120$ degrees}
\end{center}
\end{figure}
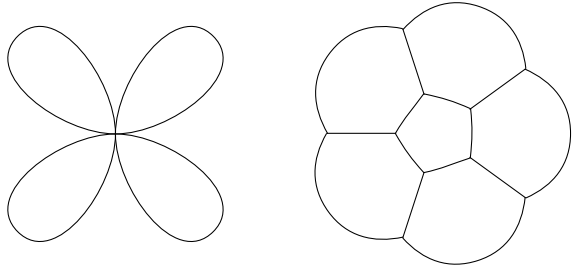

Then, the problem we are interested in is the following:
\begin{problem}\label{Pgeneral}
Is
\begin{align*}
\inf\lbrace F\left(\Gamma\right)\vert\; &\Gamma\;\text{is a regular $N$--network of class}\; H^2 \\
&\text{whose curves meet in junctions with \emph{fixed} angles}\rbrace\,,
\end{align*}
attained?
\end{problem}

A simplification of the problem could be requiring that:
\begin{enumerate}
\item the angles at the junctions
are  fixed  in such a way that  the sum of the unit tangent vectors of the joining curves at the junctions
is equal to zero;
\item the curves meet only in triple junctions with angles of $120$ degrees
(a particular case of the previous assumption).
\end{enumerate} 
These are possible choices, not really justified by variational evidence.

Moreover we want to mention that to penalize the length and fix 
the angles is not the only reasonable possibility in this context. For instance 
one could also fix the length of each curve separately in an appropriate way 
that ensures a lower bound on the energy.

\subsection{Theta--networks}\label{theta}

In~\cite{danovplu} we study a particular case of Problem~\ref{Pgeneral}
for regular $3$--networks.
\begin{defn}
A Theta--network is a regular $3$-network of class $H^2$ such that the three 
curves 
$\gamma^i:\left[0,1\right]\rightarrow \mathbb{R}^2$, $i\in\{1,2,3\}$, meet in two triple junctions
forming angles of $120$ degrees.
\end{defn}

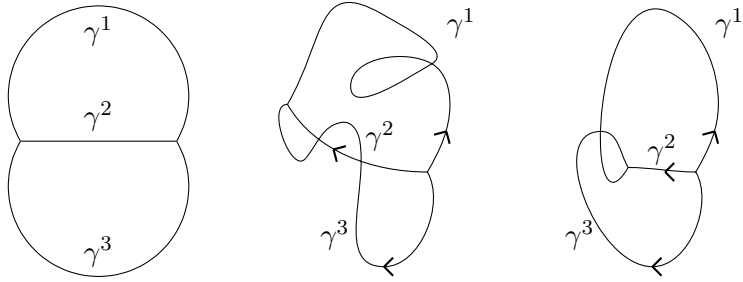
\begin{figure}[H]
\begin{center}
\begin{tikzpicture}[scale=1.2]
\draw[color=black,scale=1,domain=-2.09: 2.09,
smooth,variable=\t,shift={(0,0)},rotate=0]plot({1.*sin(\t r)},
{1.*cos(\t r)}) ; 
\draw[color=black,scale=1,domain=-2.09: 2.09,
smooth,variable=\t,shift={(0,-1)},rotate=180]plot({1.*sin(\t r)},
{1.*cos(\t r)}) ; 
\draw
(-0.87,-0.5)--(0.87,-0.5);
\path[shift={(0,0)}] 
 (0,-0.5)[above] node{$\gamma^2$}
 (0,1)[below] node{$\gamma^1$}
 (0,-2)[above] node{$\gamma^3$};
\end{tikzpicture}\qquad\;
\begin{tikzpicture}[scale=1.2]
\draw[shift={(0,0)}] 
(-1.73,-1.8)
to[out= 180,in=-80, looseness=1] 
(-2,-0.4) 
to[out= 100,in=10, looseness=1] 
(-2.2,-0.2) 
to[out= -170,in=50, looseness=1] 
(-2.6,-0.6) 
to[out= -130,in=180, looseness=1] 
(-2.8,0) 
to[out= 60,in=150, looseness=1.5] (-1.75,1) 
(-2.8,0)
to[out=-60,in=180, looseness=0.9] (-1.25,-0.75)
(-1.75,1)
to[out= -30,in=30, looseness=0.9](-1.2,0.45)
(-2.1,0.2)
to[out= -90,in=-150, looseness=0.9] 
(-1.2,0.45)
(-2.1,0.2)
to[out= 90,in=150, looseness=0.9] 
(-1.2,0.45)
to[out=-30,in=90, looseness=0.9] (-1,0)
to[out= -90,in=60, looseness=0.9] (-1.25,-0.75)
to[out= -60,in=0, looseness=0.9](-1.73,-1.8);
\path[shift={(0,0)}] 
 (-1.5,-0.35)[left] node{$\gamma^2$}
 (-0.6,.9)[left] node{$\gamma^1$}
 (-2,-1.45)[left] node{$\gamma^3$};
  \draw[thick, scale=1, shift={(-2.29,-0.5)}, rotate=60]
(0,0)to[out= -45,in=135, looseness=1] (0.1,-0.1)
(0,0)to[out= -135,in=45, looseness=1] (-0.1,-0.1);
\draw[ thick, shift={(-1.04,-0.3)}, scale=1, rotate=-30]
(0,0)to[out= -45,in=135, looseness=1] (0.1,-0.1)
(0,0)to[out= -135,in=45, looseness=1] (-0.1,-0.1);
\draw[thick, shift={(-1.73,-1.8)}, scale=1, rotate=90]
(0,0)to[out= -45,in=135, looseness=1] (0.1,-0.1)
(0,0)to[out= -135,in=45, looseness=1] (-0.1,-0.1);
 \end{tikzpicture}\qquad
\begin{tikzpicture}[scale=1.2]
\draw[shift={(0,0)}] 
(-2.3,-0.3) 
to[out= 0,in=120, looseness=1] (-2,-0.7) 
to[out= -120,in=150, looseness=1.5] (-1.5,1) 
(-2,-0.7)
to[out=0,in=180, looseness=0.9] (-1.25,-0.75)
(-1.5,1)
to[out= -30,in=90, looseness=0.9] (-1,0)
to[out= -90,in=60, looseness=0.9] (-1.25,-0.75)
to[out= -60,in=0, looseness=0.9](-1.73,-1.8)
(-1.73,-1.8)to[out=180,in=180, looseness=1] (-2.3,-0.3) ;
 \draw[thick, scale=1, shift={(-1.59,-0.75)}, rotate=90]
(0,0)to[out= -45,in=135, looseness=1] (0.1,-0.1)
(0,0)to[out= -135,in=45, looseness=1] (-0.1,-0.1);
\draw[ thick, shift={(-1.04,-0.3)}, scale=1, rotate=-30]
(0,0)to[out= -45,in=135, looseness=1] (0.1,-0.1)
(0,0)to[out= -135,in=45, looseness=1] (-0.1,-0.1);
\draw[thick, shift={(-1.73,-1.8)}, scale=1, rotate=90]
(0,0)to[out= -45,in=135, looseness=1] (0.1,-0.1)
(0,0)to[out= -135,in=45, looseness=1] (-0.1,-0.1);
\path[shift={(0,0)}] 
 (-1.35,-0.5)[left] node{$\gamma^2$}
 (-0.6,.9)[left] node{$\gamma^1$}
 (-2.26,-1.45)[left] node{$\gamma^3$};
 \end{tikzpicture}\qquad
\end{center}
\caption{Three examples of Theta--networks.}
\end{figure}

Notice that we do not require nor the curve of a Theta--network to be injective
neither the network to be embedded, hence intersections between the curves or
self--intersections of the curves can occur.

\begin{prob}\label{Pdoppiabolla}
Does a minimizer to $\inf\lbrace F\left(\Gamma\right)\vert\; \Gamma\;
\text{is a Theta--network}\;\rbrace\,$ 
exists?
\end{prob}

We discuss here  in an informal way the result given in~\cite{danovplu} 
on the existence and characterization of minimizers for Problem~\ref{Pdoppiabolla}.  
With an estimate in the style of~\eqref{a},~\eqref{b} 
we see in~\cite[Remark 4.13]{danovplu} that for every Theta--network $\Gamma$
$$
F(\Gamma)\geq 4\pi\,.
$$
Now, consider an equibounded sequence $\Gamma_n$
of Theta--networks, then there exists (up to subsequence) 
an $H^2$--weak limit $\Gamma_\infty$
of $\Gamma_n$.
Again, with arguments based on Gauss--Bonnet Theorem, in~\cite{danovplu} it is shown that
the length of (at most) one of the three curves of $\Gamma_n$
can go to zero as $n\to\infty$. In other words 
the class of Theta--networks 
is not closed.
In order to prove 
a theorem about the existence of minimizers for our problem,
we 
introduce
the class of \textit{``degenerate'' 
Theta--networks}.

\begin{defn} 
A  \textit{``degenerate'' 
Theta--network} is a network  composed by two regular curves $\gamma^1,\gamma^2$
of class $H^2$,
forming angles in pairs of $120$ and $60$
degrees
and by a curve $\gamma^3$ of length zero.
\end{defn}

Then one extends the functional $F$ to a new functional $\overline{F}$
defined 
on all 
$3$--networks $\Gamma$ of class $H^2$ 
as follows
$$
\overline{F}(\Gamma)=
\begin{cases}
\sum_{i=1}^3
\int_{\gamma^i}\big((k^i)^2+1\big)\,{\rm{d}}s^i \qquad \text{if}\; \Gamma
\;\text{is Theta--network, }\\
\sum_{i=1}^2
\int_{\gamma^i}\big((k^i)^2+1\big)\,{\rm{d}}s^i \qquad \text{if}\; \Gamma
\;\text{is a ``degenerate'' 
Theta--network, }\\
+\infty \qquad\qquad\quad\,\;\;\,\quad\qquad\quad\,\,\, \text{otherwise}\,.
\end{cases}
$$
In~\cite[Theorem 4.7 and Corollary 4.10]{danovplu} 
we prove  existence of minimizer for $\overline{F}$
and we show that $\overline{F}$ is the relaxation of $F$.

\begin{rem}
Due to the analysis done in~\cite{danovplu} we expect that also in Problem~\ref{Pgeneral} 
the class of regular  $N$--network is not closed. 
In that case the appropriate notion of ``degenerate'' network has still to be understood. 
We expect these to be networks with $N-k$ curves for some value(s) $k\in\{1,\ldots,N-1\}$, 
but furthermore one needs an understanding of which types
of networks arise as limits of minimizing sequences. For instance saying
for how many curves the length can go to zero and, in this case, 
which angles are present at the junctions in the limit. 
This kind of questions will be investigated in a forthcoming paper.
\end{rem}

\begin{rem}\label{altriangoli}
Coming back to the case of 
$3$--networks of class $H^2$,
we have already mentioned that fixing the angles is a crucial point
to get a non trivial problem. 
Let us call $\tau^i$ the unit tangent vector to the $i$--th curve at the junctions,
$\alpha_1$ the angle between $\tau^1$ and $\tau^2$, $\alpha_2$ the angle between $\tau^2$ and $\tau^3$ and $\alpha_3$ the angle between $\tau^3$
and $\tau^1$.
For every triple of angles $(\alpha_1,\alpha_2,\alpha_3)$ such that $\alpha_1+\alpha_2+\alpha_3=2\pi$ and different from $(0,0,2\pi)$ (or $(0,2\pi,0), (2\pi,0,0)$) 
the compactness and lower semicontinuity results presented in~\cite[Section 4]{danovplu}
remain true. The choice done in the case of Theta--networks of
all equal angles 
is not really motivated variationally but it simplifies 
a bit the boundary terms in the computation of the Euler-Lagrange equations 
(see Section~\ref{secEL}).
\end{rem}

\begin{figure}[h]
\begin{center}
\begin{tikzpicture}[scale=1.7]
\draw[color=black,scale=0.05,domain=0.7: 2.4,
smooth,variable=\t,shift={(40,0)},rotate=0]plot({2.*sin(\t r)},
{2.*cos(\t r)}) ; 
\draw[color=black,scale=0.05,domain=-1.35: 0.35,
smooth,variable=\t,shift={(40,0)},rotate=0]plot({2.*sin(\t r)},
{2.*cos(\t r)}) ; 
\draw[color=black,scale=0.05,domain=2.9: 3.15,
smooth,variable=\t,shift={(40,0)},rotate=0]plot({2.*sin(\t r)},
{2.*cos(\t r)}) ; 
\draw[color=black,scale=0.05,domain=-3.15: -1.85,
smooth,variable=\t,shift={(40,0)},rotate=0]plot({2.*sin(\t r)},
{2.*cos(\t r)}) ; 
\path[font=\small]
(2,-0.7)[below] node{$\alpha_1=\alpha_2=\alpha_3=\frac{2\pi}{3}$}
(2.5,0)[left] node{$\alpha_3$}
(1.7,0.25)[right] node{$\alpha_1$}
(1.7,-0.25)[right] node{$\alpha_2$}
(2.65,-0.55)[left] node{$\tau^3$}
(2.65,0.55)[left] node{$\tau^1$}
(1.41,0)[left] node{$\tau^2$};
\draw[thick, scale=0.5, shift={(4,0)}]
(0,0)--(0.5,-0.86)
(0,0)--(0.5,0.86)
(0,0)--(-1,0);
 \draw[thick,  shift={(2.25,0.43)}, scale=0.5, rotate=-30]
(0,0)to[out= -45,in=135, looseness=1] (0.1,-0.1)
(0,0)to[out= -135,in=45, looseness=1] (-0.1,-0.1);
\draw[ thick,  shift={(2.25,-0.43)}, scale=0.5, rotate=-150]
(0,0)to[out= -45,in=135, looseness=1] (0.1,-0.1)
(0,0)to[out= -135,in=45, looseness=1] (-0.1,-0.1);
\draw[thick, shift={(1.5,0)}, scale=0.5, rotate=90]
(0,0)to[out= -45,in=135, looseness=1] (0.1,-0.1)
(0,0)to[out= -135,in=45, looseness=1] (-0.1,-0.1);
\end{tikzpicture}\qquad\qquad
\begin{tikzpicture}[scale=1.7]
\draw[color=black,scale=0.05,domain=1.1: 2.85,
smooth,variable=\t,shift={(40,0)},rotate=0]plot({2.*sin(\t r)},
{2.*cos(\t r)}) ; 
\draw[color=black,scale=0.05,domain=-1.35: 0.5,
smooth,variable=\t,shift={(40,0)},rotate=0]plot({2.*sin(\t r)},
{2.*cos(\t r)}) ; 
\draw[color=black,scale=0.05,domain=-2.85: -1.85,
smooth,variable=\t,shift={(40,0)},rotate=0]plot({2.*sin(\t r)},
{2.*cos(\t r)}) ; 
\path[font=\small]
(2,-0.7)[below] node{$\alpha_1=\alpha_3=\frac{3\pi}{4}\,,\alpha_2=\frac{\pi}{2}$}
(2.5,-0.1)[left] node{$\alpha_3$}
(1.78,0.25)[right] node{$\alpha_1$}
(1.55,-0.25)[right] node{$\alpha_2$}
(2.4,-0.55)[left] node{$\tau^3$}
(2.7,0.5)[left] node{$\tau^1$}
(1.41,0)[left] node{$\tau^2$};
\draw[thick, scale=0.5, shift={(4,0)}]
(0,0)--(0,-1)
(0,0)--(0.7,0.7)
(0,0)--(-1,0);
 \draw[thick,  shift={(2.36,0.36)}, scale=0.5, rotate=-45]
(0,0)to[out= -45,in=135, looseness=1] (0.1,-0.1)
(0,0)to[out= -135,in=45, looseness=1] (-0.1,-0.1);
\draw[ thick,  shift={(2,-0.5)}, scale=0.5, rotate=-180]
(0,0)to[out= -45,in=135, looseness=1] (0.1,-0.1)
(0,0)to[out= -135,in=45, looseness=1] (-0.1,-0.1);
\draw[thick, shift={(1.5,0)}, scale=0.5, rotate=90]
(0,0)to[out= -45,in=135, looseness=1] (0.1,-0.1)
(0,0)to[out= -135,in=45, looseness=1] (-0.1,-0.1);
\end{tikzpicture}\qquad\qquad
\begin{tikzpicture}[scale=1.7]
\draw[color=black,scale=0.05,domain=-1.35: 3.15,
smooth,variable=\t,shift={(40,0)},rotate=0]plot({2.*sin(\t r)},
{2.*cos(\t r)}) ; 
\draw[color=black,scale=0.05,domain=-3.15: -1.85,
smooth,variable=\t,shift={(40,0)},rotate=0]plot({2.*sin(\t r)},
{2.*cos(\t r)}) ; 
\path[font=\small]
(1.3,-0.7)[below] node{$\alpha_1=\alpha_2=0\,,\alpha_3=2\pi$}
(2.5,0)[left] node{$\alpha_3$}
(1.41,0)[left] node{$\tau^1\equiv\tau^2\equiv\tau^3$};
\draw[thick, scale=0.5, shift={(4,0)}]
(0,0)--(-1,0);
\draw[thick, shift={(1.5,0)}, scale=0.5, rotate=90]
(0,0)to[out= -45,in=135, looseness=1] (0.1,-0.1)
(0,0)to[out= -135,in=45, looseness=1] (-0.1,-0.1);
\end{tikzpicture}
\end{center}
\caption{Three examples of triples of angles: the angles of a Theta--network,
an admissible generalized triple and a not admissible one.
}
\label{diffangoli}
\end{figure}
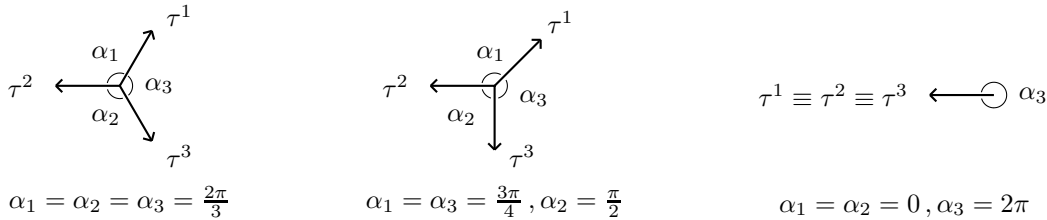

Apart from proving an existence theorem, 
our main aim in~\cite{danovplu} is to characterize the minimizers.
In particular we 
show that the minimizer for the relaxed functional is in fact a Theta--network. 
We give here the main ideas.
\begin{enumerate}
\item Firstly we prove that the energy of any ``degenerate''  Theta--network
is greater or equal to the energy 
of the  ``Figure Eight" $\mathcal{F}$ (see Figure~\ref{otto}). 
We cannot simply say that the ``Figure Eight" is the optimal ``degenerate''  
Theta--network since it has the ``wrong'' angles. 
A  ``degenerate''   Theta--network 
has angles of 60 and 120 degrees at the junctions whereas 
in~\cite{explicitelasticae} it is computed that the ``Figure Eight" has angles of $\approx 80$ and $\approx 100$ degrees at the junctions.
\item The second step of the proof is based on the computations given in full details here in Section~\ref{bulgari} 
to get an approximate value for the energy of the ``Figure Eight". We get that $F(\mathcal{F})\approx 21.2075$.
\item Then to conclude it is enough to exhibit a Theta--network
with strictly less energy than $\mathcal{F}$. This is achieved by computing the energy 
of the standard double
bubble which is 
$$
\frac{2}{3}\sqrt{8\pi(8\pi+3\sqrt{3})}\approx 18.4059\,.
$$
\end{enumerate}

{\it{Generalization to different angles}}

Writing this paper we have noticed that the proof that we have just discussed 
can be indeed applied, 
as we are going to explain, 
 to more general networks
than to Theta--networks. Consider indeed a network that satisfies the generalized condition on the
 angles presented in Remark \ref{altriangoli}. Then also in this case, with arguments
 based on Gauss-Bonnet, along a minimizing sequence
  the length of at most one of the three curves can go to zero. 
We introduce the class  of  \emph{double drops}~\cite[Definition 3.6]{danovplu}: 
networks composed 
by two drops (defined as in~\ref{defdrop}) that meet at a four point $P$
forming \textit{any possible angle} at $P$.
Then one notice that ``degenerate''  Theta--networks
can be 
seen as the union of two drops
$\gamma$ and $\tilde{\gamma}$  that meet at
a four--point $\gamma^1(0)=\gamma^1(1)=\gamma^2(0)=\gamma^2(1)=:P$
forming angles in pairs of $120$ and $60$ degrees at $P$.
Among all double drops a minimizer is the
``Figure Eight" and 
hence a fortiori all ``degenerate''  Theta--networks have energy bigger 
or equal than the energy of the  ``Figure Eight".
Changing the triple of angles in the definition of Theta--network (see Remark~\ref{altriangoli}), 
the class of ``degenerate''  Theta--networks will change too, but it
remains true that the energy of the ``Figure Eight"
is less or equal to the energy of any new ``degenerate''  network.

\medskip

This part of the proof is not sharp in the sense that the energy of 
any  ``degenerate''  network with angles of $120$ and $60$
degrees in pair (or $\alpha_1$ and $\alpha_2$
different from the ones of the ``Figure Eight" 
 in the general case)
could be much higher than the energy of the
``Figure Eight".

\medskip

The last step of the proof depends on a quantitative inequality:
the energy of 
the standard double bubble, easy to compute explicitly,
 is less than the energy of 
  the ``Figure Eight".
We show here that for any triple of angles $(\alpha_1,\alpha_2,\alpha_3)$ 
with all angles different from zero and with two angles less or equal then 135 degrees, 
the standard double bubble can be replaced by a network $\mathcal{N}$ composed
by a segment $\mathcal{S}$ and two arcs of circle
$\mathcal{A}_1,\mathcal{A}_2$
 meeting with angles $\alpha_1,\alpha_2,\alpha_3$
and the energy of this network is still less than the energy of 
 the ``Figure Eight". In particular, also the last step of the proof works in this more general case.

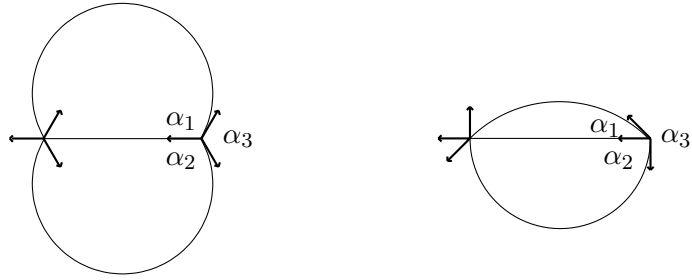
\begin{figure}[h]
\begin{center}
\begin{tikzpicture}[scale=1.2]
\draw[color=black,scale=1,domain=-2.09: 2.09,
smooth,variable=\t,shift={(0,0)},rotate=0]plot({1.*sin(\t r)},
{1.*cos(\t r)}) ; 
\draw[color=black,scale=1,domain=-2.09: 2.09,
smooth,variable=\t,shift={(0,-1)},rotate=180]plot({1.*sin(\t r)},
{1.*cos(\t r)}) ; 
\draw
(-0.87,-0.5)--(0.87,-0.5);
\path[shift={(0,0)}] 
 (0.65,-0.95)[above] node{$\alpha_2$}
 (0.65,-0.1)[below] node{$\alpha_1$}
 (1,-0.5)[right] node{$\alpha_3$};
  \draw[thick, scale=0.35, shift={(-2.5,-1.425)}]
(0,0)--(0.5,-0.86)
(0,0)--(0.5,0.86)
(0,0)--(-1,0);
 \draw[thick, scale=0.35, shift={(2.5,-1.425)}]
(0,0)--(0.5,-0.86)
(0,0)--(0.5,0.86)
(0,0)--(-1,0);
 \draw[thick,  shift={(1.065,-0.19)}, scale=0.35, rotate=-30]
(0,0)to[out= -45,in=135, looseness=1] (0.1,-0.1)
(0,0)to[out= -135,in=45, looseness=1] (-0.1,-0.1);
\draw[ thick,  shift={(1.065,-0.81)}, scale=0.35, rotate=-150]
(0,0)to[out= -45,in=135, looseness=1] (0.1,-0.1)
(0,0)to[out= -135,in=45, looseness=1] (-0.1,-0.1);
 \draw[thick,  shift={(-0.685,-0.19)}, scale=0.35, rotate=-30]
(0,0)to[out= -45,in=135, looseness=1] (0.1,-0.1)
(0,0)to[out= -135,in=45, looseness=1] (-0.1,-0.1);
\draw[ thick,  shift={(-0.685,-0.81)}, scale=0.35, rotate=-150]
(0,0)to[out= -45,in=135, looseness=1] (0.1,-0.1)
(0,0)to[out= -135,in=45, looseness=1] (-0.1,-0.1);
\draw[thick, shift={(0.5,-0.5)}, scale=0.35, rotate=90]
(0,0)to[out= -45,in=135, looseness=1] (0.1,-0.1)
(0,0)to[out= -135,in=45, looseness=1] (-0.1,-0.1);
\draw[thick, shift={(-1.25,-0.5)}, scale=0.35, rotate=90]
(0,0)to[out= -45,in=135, looseness=1] (0.1,-0.1)
(0,0)to[out= -135,in=45, looseness=1] (-0.1,-0.1);
\end{tikzpicture}\qquad\qquad\qquad
\begin{tikzpicture}[scale=1.2]
\draw[color=black,scale=1,domain=-1.57: 1.57,
smooth,variable=\t,shift={(-1,0)},rotate=180]plot({1.*sin(\t r)},
{1.*cos(\t r)}) ; 
\draw
(-2,0)--(0,0);
\draw[thick,shift={(-2,0)}, scale=0.35]
(0,0)--(0,1)
(0,0)--(-0.7,-0.7)
(0,0)--(-1,0);
\draw[thick, scale=0.35]
(0,0)--(0,-1)
(0,0)--(-0.7,0.7)
(0,0)--(-1,0);
 \draw[thick,  shift={(-0.25,0.25)}, scale=0.35, rotate=45]
(0,0)to[out= -45,in=135, looseness=1] (0.1,-0.1)
(0,0)to[out= -135,in=45, looseness=1] (-0.1,-0.1);
\draw[ thick,  shift={(0,-0.35)}, scale=0.35, rotate=-180]
(0,0)to[out= -45,in=135, looseness=1] (0.1,-0.1)
(0,0)to[out= -135,in=45, looseness=1] (-0.1,-0.1);
\draw[thick, shift={(-0.35,0)}, scale=0.35, rotate=90]
(0,0)to[out= -45,in=135, looseness=1] (0.1,-0.1)
(0,0)to[out= -135,in=45, looseness=1] (-0.1,-0.1);
\draw[color=black,scale=1.405,domain=-0.78: 0.78,
smooth,variable=\t,shift={(-0.72,-0.71)},rotate=0]plot({1.*sin(\t r)},
{1.*cos(\t r)}) ; 
 \draw[thick,  shift={(-2.25,-0.25)}, scale=0.35, rotate=135]
(0,0)to[out= -45,in=135, looseness=1] (0.1,-0.1)
(0,0)to[out= -135,in=45, looseness=1] (-0.1,-0.1);
\draw[ thick,  shift={(-2,0.35)}, scale=0.35, rotate=0]
(0,0)to[out= -45,in=135, looseness=1] (0.1,-0.1)
(0,0)to[out= -135,in=45, looseness=1] (-0.1,-0.1);
\draw[thick, shift={(-2.35,0)}, scale=0.35, rotate=90]
(0,0)to[out= -45,in=135, looseness=1] (0.1,-0.1)
(0,0)to[out= -135,in=45, looseness=1] (-0.1,-0.1);
\draw[white]
(-1,-1.5)--(0,-1.5);
\path[shift={(-1,0.5)}] 
 (0.65,-0.95)[above] node{$\alpha_2$}
 (0.5,-0.18)[below] node{$\alpha_1$}
 (1,-0.5)[right] node{$\alpha_3$};
\end{tikzpicture}
\end{center}
\caption{On the left the standard double bubble and on the right the double bubble with angles $(\alpha_1,\alpha_2,\alpha_3)$.}\label{bubble}
\end{figure}

Suppose w.l.o.g. that $0<\alpha_1\leq\alpha_2\leq\alpha_3$. 
Consider 
a segment $\mathcal{S}$ of length $2 \sin(\alpha_1)R$ with $R>0$
and let $\mathcal{A}_1$ be the circular arc that forms an angle $\alpha_1$ 
with the segment $\mathcal{S}$ at both end points. 
Let $\mathcal{A}_2$ be the circular arc such that between 
$\mathcal{S}$ and $\mathcal{A}_2$ there is an angle equal to $\alpha_2$ at both end-points. 
Then $\mathcal{N}=\{\mathcal{A}_1,\mathcal{A}_2,\mathcal{S}\}$ is a regular 
$3$--network of class $H^2$ with 
angles $(\alpha_1,\alpha_2,\alpha_3)$.
It is easy to see that 
$$
\begin{array}{lll}
 L(\mathcal{A}_1)=2\alpha_1 R\,,\quad&
L(\mathcal{A}_2)=2\alpha_2\frac{\sin\alpha_1}{\sin\alpha_2} R\,,\quad&
L(\mathcal{S})=2\sin\alpha_1 R\,,\\
E(\mathcal{A}_1)=\frac{2\alpha_1 }{R}\,,&
E(\mathcal{A}_2)=\frac{2\alpha_2\sin\alpha_2}{\sin\alpha_1R}\,,&
E(\mathcal{S})=0\,.\\
\end{array}
$$
and that
$F(\mathcal{N})=\frac{2}{R}\left(\alpha_1+\frac{\alpha_2\sin\alpha_2}{\sin\alpha_1}
\right)+
2R\left(\alpha_1+ \alpha_2\frac{\sin\alpha_1}{\sin\alpha_2}+\sin\alpha_1\right)$.
Since  $R$ is a free parameter, we optimize now over $R$:
by a direct computation, 
the energy of the optimal rescaling $\mathcal{N}_\mathrm{opt}$  of $\mathcal{N}$ is
$$
F(\mathcal{N}_\mathrm{opt})
=4\sqrt{\alpha_1+\frac{\alpha_2\sin\alpha_2}{\sin\alpha_1}}
\sqrt{\alpha_1+ \alpha_2\frac{\sin\alpha_1}{\sin\alpha_2}+\sin\alpha_1}\,.
$$
If $\alpha_1\leq\alpha_2\leq \frac{3\pi}{4}$, then $F(\mathcal{N}_\mathrm{opt})<21.2075$,
hence we can conclude again that the minimizers are not ``degenerate''.

\medskip

Notice that  the condition on the angles gives rise to an easy generalization of the 
last step of the proof, but 
we do not expect that the condition  is optimal.

\medskip

In~\cite{danovplu} we have also proved that the minimizers are injective, 
in the sense that each curve of a minimal network is injective. 
It remains to prove that they are embedded: also even though we know that each curve in the minimizing network
has no self--intersections, a curve can  still intersect the others.
We underline that in the case of minimization of the elastic energy for networks many arguments that work when minimizing the length, fail here or are not applicable. 
For example the classical convexification argument works only partially
because of the presence of the angles at the junctions. Symmetrization arguments are difficult 
due to lack of comparison principles and at each step of the proof
you have to take care of the fact
 that each curve in the network has to be in $H^2$ 
and hence when replacing parts of the curves continuity of the tangent has to be ensured.

\subsection{Open networks}
In the Definition~\ref{network} of $N$--networks 
we require that both end points of each curve meet
with other curves  in junctions. Similar to the case of open curves, one could define \textit{``open" networks} requiring that some of the end points 
of the curves do not meet in junction, but are fixed. 

\subsection{Elastic clusters}

Consider a sufficiently smooth embedded regular $N$--network. Due to the embeddedness, 
each curve of the network is injective and does intersect the other curves only at the end points.
Such networks define a partition of the plane in finitely many bounded sets
$E_1,  \ldots, E_n$ and an unbounded one
$E_0:=\mathbb{R}^2\setminus\cup_{i=1}^n E_i$.
We can ask 
the sets $E_i$ to be open, connected and piecewise regular.
We call the sets $E_i$ 
chambers or bubbles and we denote their area by $m_i$, i.e. $\vert E_i\vert=m_i$
with $i\in\{1,\ldots, n\}$.
The family
$\mathcal{E}=(E_1, \ldots, E_n)$
is usually called $n$--cluster. 
An interesting problem would be to 
minimize the functional $F$, defined in~\eqref{funzionalenetworks},
between all networks that 
give an $n$--cluster
with the area $m_i$
of each chamber fixed.


\subsection{Extensions to $\mathbb{R}^n$}

In Definition~\ref{network}
we consider $N$--network composed by planar curves. It is also possible to
change the ambient space from $\mathbb{R}^2$ to  $\mathbb{R}^n$
taking curves $\gamma^i:[0,1]\to\mathbb{R}^n$
and rephrase the previous problems in this framework.
In this direction 
in the next section the computation of  the Euler Lagrange equations for 
Problem~\ref{problemdrops} and Problem~\ref{Pdoppiabolla}
are directly applicable in $\mathbb{R}^n$.

\section{Euler-Lagrange equations}\label{secEL}

In this section 
we derive  the Euler--Lagrange equations
for the energy functional $F_{\delta}$, $\delta>0$, given in \eqref{funzionalenetworks},
for the curves $\gamma^i$ of the $N$--networks in terms of the curvature vector.
The equations can be directly generalized to 
 $\mathbb{R}^n$. 
 
We define the operator $\partial_s^{\perp}$ on vector fields as the normal component of $\partial_s$. That is, 
$$ \partial_s^{\perp} \phi := \partial_s \phi- \langle \partial_s \phi, \partial_s \gamma\rangle \partial_s \gamma \, . $$
In the literature this operator is often denoted by $\nabla_s$.

\begin{prop}[Euler-Lagrange equation]
Fixed $\delta>0$, if a regular $N$--network $\Gamma$  
is a stationary point for $F_\delta$, then
each curve $\gamma^i$ of $\Gamma$ is of class $C^\infty$ and satisfies
\begin{equation}\label{eulervector}
2 (\partial^\perp_s)^2 \boldsymbol{\kappa}^{i}+ |\boldsymbol{\kappa}^i|^2 \boldsymbol{\kappa}^{i} -\delta \boldsymbol{\kappa}^{i}=0\, \mbox{ on }(0,1),
\end{equation}
or in terms of the scalar curvature 
\begin{equation}\label{eulero}
2 \partial^2_{s} k^{i}+\left(k^{i}\right)^{3} -\delta k^{i}=0 \mbox{ on }(0,1)\,.
\end{equation}
Moreover if we consider the problem defined only for networks with triple junctions
and we fix all equal angles 
at the triple junctions  the following conditions are fulfilled:
\begin{equation*}
\sum_{i=1}^3k^i=0\qquad\text{and}\qquad \sum_{i=1}^3 2 \partial_{s}k^i\ \nu^i+(k^i)^2\tau^i =0\,,
\end{equation*}
at the junctions. 
\end{prop}

\begin{proof}
\textit{First variation} 

Let $\Gamma$ be a regular connected $N$--network of class $H^2$. 
Since the number of junctions is finite
it is sufficient to consider the variation of $M$ curves meeting at a junction, $2\leq M\leq N$. 
Let $O$ be the junction
around which we variate and, without loss generality, let  
$\gamma^i:[0,1]\to\mathbb{R}^2$, $i\in\{1,...,M\}$, 
be the regular curves meeting at $O$. 
Then, $\gamma^{i}(0)=O$ for  $i\in\{1,...,M\}$ and possibly also $\gamma^j(1)=O$ 
for some values of the parameter $j \in \{1, ..,M\}$. Indeed, in a general (connected) network
some of the curves can start end ends at the same junction point.

We want to take $\widetilde{\Gamma}$
a variation of $\Gamma$.
To this aim we 
take $t\in\mathbb{R}$ and 
curves $\psi^{i}:[0,1]\to\mathbb{R}^2$, $i\in\{1,...,M\}$, 
in $C^\infty([0,1])$
such that $\psi^{i}(0)=\psi^{j}(0)$ for $i,j\in\{1,...,M\}$. 
Moreover, $\psi^i \in C^{\infty}_0([0,1))$ if $\psi^i(1)\ne O$,
otherwise  
we ask that $\psi^i(0)=\psi^i(1)$. 
Let $\tilde{\Gamma}$ be the network consisting of the curves 
$\tilde{\gamma}_i=\gamma_i + t \psi_i$, $i\in\{1,...,M\}$, 
and (if $M<N$) the curves $\tilde{\gamma}_i=\gamma_i$ for $i\in\{M+1, ..., N\}$. 
For $|t|$ small enough, we have that $\widetilde{\Gamma}$ is still a regular network 
with the property that $O$ is a junction point for the curves $\tilde{\gamma}^i$ with $i\in\{1,..,M\}$, 
and at the other junctions $\tilde{\Gamma}$ is equal to $\Gamma$. 

We compute now the first variation 
$\frac{d}{dt}F_{\delta}(\widetilde{\Gamma})_{|t=0}$. 
From now on we omit the dependence on the variable $x$ to simplify the notation.
Using the definition of curvature vector~\eqref{curvvector} and
making use of the fact that $\boldsymbol{\kappa}^i$ is normal to $\gamma^i$, we have
\begin{align}\nonumber
 \frac{d}{dt}F_{\delta}(\widetilde{\Gamma})_{|t=0}
 & =\sum_{i=1}^{M}\int_{0}^{1} 2 \Big\langle \boldsymbol{\kappa}^i,
 \frac{\partial^2_{x}\psi^{i}}{\left|\partial_x\gamma^{i}\right|^{2}}
 -\frac{\partial_x\psi^{i}}{\left|\partial_x\gamma^{i}\right|^{4}}\langle \partial_x\gamma^{i}, 
 \partial^2_{x}\gamma^{i}\rangle \Big\rangle\left|\partial_x\gamma^{i}\right| \mathrm{d}x \\ \nonumber
& \quad + \sum_{i=1}^{M} \int_{0}^{1} (-3|\boldsymbol{\kappa}^i|^2+\delta)
 \frac{\left\langle \partial_x\psi^{i},\partial_x\gamma^{i}\right\rangle }{\left|\partial_x\gamma^{i}\right|}\,\mathrm{d}x \\ \label{primaver}
 & =\sum_{i=1}^{M} \left[\int_{0}^{1} 2 \langle \boldsymbol{\kappa}^{i},
 \partial^2_{s} \psi^{i} \rangle \, \mathrm{d}s   + \int_{0}^{1} (-3 |\boldsymbol{\kappa}^i|^2+\delta)
\left\langle \tau^{i},\partial_s\psi^{i}\right\rangle \,\mathrm{d}s^i \right]\,.
\end{align}

\textit{Euler-Lagrange equations}\\
Let us assume now that the curves $\gamma^i$ are in $H^4$. 
Then, integrating by parts
the term where the test functions $\psi^{i}$ 
appear derived twice we get
\begin{align*}
 \frac{d}{dt}F_{\delta}(\widetilde{\Gamma})_{|t=0}
 & =\sum_{i=1}^{M} 2 \left. \langle \boldsymbol{\kappa}^{i},
 \partial_s \psi^{i}\rangle \right|_0^1 + \sum_{i=1}^{M} \left[\int_{0}^{1}  \left\langle 
 -2\partial_s \boldsymbol{\kappa}^{i}+ (-3 |\boldsymbol{\kappa}^i|^2+\delta)\tau^{i},
\partial_s\psi^{i}\right\rangle \,\mathrm{d}s \right]\,.
\end{align*}
Integrating by parts once more and using the fact that 
$\partial_s \boldsymbol{\kappa}=\partial^\perp_s \boldsymbol{\kappa}
- |\boldsymbol{\kappa}|^2\tau$, we get
\begin{align}\label{secondintegration}
 &\frac{d}{dt}F_{\delta}(\widetilde{\Gamma})_{|t=0}
 =\sum_{i=1}^{M} \left[ 2 \left. \langle \boldsymbol{\kappa}^{i},
 \partial_s \psi^{i}\rangle \right|_0^1 
+  \left. \langle -2\partial^\perp_s \boldsymbol{\kappa}^{i}
- |\boldsymbol{\kappa}^i|^2\tau^i+\delta\tau^i, \psi^{i}\rangle \right|_0^1 
 \right] \nonumber \\
& \quad 
+ \sum_{i=1}^{M} 
 \int_{0}^{1}  \left\langle
 2 \partial_s^2 \boldsymbol{\kappa}^{i}
 +  (3 |\boldsymbol{\kappa}^i|^2-\delta)\boldsymbol{\kappa}^{i}
 +6 \langle \boldsymbol{\kappa}^i, \partial_s \boldsymbol{\kappa}^i\rangle \tau^i,
\psi^{i} \right\rangle \,\mathrm{d}s^i \,.
\end{align}
Considering test functions with compact support and variating each curve separately
and using that
\begin{align*}
\partial_s^2 \boldsymbol{\kappa}^{i}= (\partial_s^\perp)^2 \boldsymbol{\kappa}^{i} - 3 \langle \partial_s \boldsymbol{\kappa}^i, \boldsymbol{\kappa}^i\rangle \partial_s \gamma^i - |\boldsymbol{\kappa}^i|^2 \boldsymbol{\kappa}^i\\
 (\partial^\perp_s)^2 \boldsymbol{\kappa}^{i} = \partial^\perp_s (\partial_ s k^{i} \;  \nu^{i}) 
 = \partial_ s^2 k^{i} \;  \nu^{i} \,,
\end{align*}
the  Euler-Lagrange equations~\eqref{eulervector} and~\eqref{eulero}
follow, respectively. 
We remind that this formula is well known in the literature, 
see for instance~\cite{AnnaPaola1}.

 \textit{Regularity}
 
Let us now show that a critical point in $H^2$ is indeed in $H^4$ and even $C^{\infty}$. In particular, we justify here the extra assumption required to derive the Euler-Lagrange equation. We are going to use the weak formulation of the Euler-Lagrange equation given in~\eqref{primaver}, taking a variation only of one curve at a time and choosing appropriate test-functions as in \cite{DDG}. 

Let us choose an index $j \in \{1, ..,M\}$, we prove the regularity of $\gamma^j$. 
We start by showing that $\boldsymbol{\kappa}^{j} \in L^{\infty}$. 

In~\eqref{primaver} we consider test functions $\psi^i\equiv 0$ for $i\ne j$ 
and $\psi^{j}:[0,1] \to \mathbb{R}^2$ defined as follows:
given $\eta_1, \eta_2 \in C^{\infty}_{0}(0,1)$ 
the components $\psi^j_1,\psi^j_2$ of $\psi^j$ are given by
$$\psi^j_l(x)= \int_0^x |\partial_x \gamma^j(y)| \left( \int_0^y \eta_l (t) \, dt \right) \, d y -x^2 (\alpha_l x +\beta_l) \, , 
$$
with
\begin{align*}
\alpha_l & = - 2\int_0^1 |\partial_x \gamma^j(y)| \left( \int_0^y \eta_l (t) \, dt \right) \, d y  +  |\partial_x \gamma^j (1)| \int_0^1 \eta_l (y) \, dy\,,  \\
\beta_l & = 3 \int_0^1 |\partial_x \gamma^j(y)| \left( \int_0^y \eta_l (t) \, dt \right) \, d y -|\partial_x \gamma^j (1)| \int_0^1 \eta_l (t) \, dt \,.
\end{align*}
for $l=1,2$. With this choice of the constants, $\psi^j \in H^2_0((0,1);\mathbb{R}^2)$. 
By the formulas above
$$ |\alpha_l|, |\beta_l| , \| \partial_s \psi^j_l\|_{L^{\infty}} \leq C \|\eta_l\|_{L^1} \,  ,$$
for some constant $C>0$ depending only on the $\delta>0$ such that
$$ 0<\delta \leq |\partial_x \gamma^j(x)|\leq \delta^{-1}<\infty \mbox{ in }[0,1]\, .$$
Here we are using the fact that the curve $\gamma^{j}$ is regular. 
Taking this special variation and choosing first $\eta_2 \equiv0$ and then $\eta_1 \equiv 0$,
it follows from~\eqref{primaver}
 that for $l=1,2$
$$ \Big|\int_{0}^{1} 2 \boldsymbol{\kappa}^{j}_l
 \eta_l \,\mathrm{d}x  \Big| \leq  C \|\eta_l\|_{L^1} \, , $$
and hence by duality that $\boldsymbol{\kappa}^{j} \in L^{\infty}$. The constant $C$ depends on $\delta$ and the elastic energy of $\gamma^j$.

We prove now that the curvature vector $\boldsymbol{\kappa}^{j}$ is weakly differentiable.
Similarly as above we variate only at the curve $\gamma^{j}$. 
Given $\eta_1, \eta_2 \in C^{\infty}_{0}(0,1)$ consider the test function $\psi^j$ 
with components $\psi^j_1,\psi^j_2$ given by
$$\psi^j_l(x)= \int_0^x |\partial_x \gamma^j(y)|  \eta_l (y) \, d y -x^2 (\alpha_l x +\beta_l)\,,\;\text{for}\;l=1,2, $$
with 
\begin{align*}
\alpha_l  = - 2\int_0^1 |\partial_x \gamma^j(y)| \eta_l (y) \, d y \,, \;
\beta_l  = 3 \int_0^1 |\partial_x \gamma^j(y)| \eta_l (y) \, d y\,,\text{for}\;l=1,2, 
\end{align*}
such that $ \psi^j \in H^2_0((0,1):\mathbb{R}^2)$. Then once again 
$$ |\alpha_l|, |\beta_l| , \| \partial_s \psi^j_l\|_{L^{1}} \leq C \|\eta_l\|_{L^1} \,  ,$$
 and reasoning as before 
 but now using also that $\boldsymbol{\kappa}^{j} \in L^{\infty}$ we find
$$\Big|\int_{0}^{1} 2 \boldsymbol{\kappa}^{j}_l
 \partial_x \eta^{j}_l  \,\mathrm{d}x \Big| \leq C \|\eta_l\|_{L^1} \, ,$$
and hence $\boldsymbol{\kappa}^{j} \in W^{1,\infty}$. 
Proceeding as in the derivation of Euler-Lagrange equations 
we can integrate by parts once and obtain that
\begin{equation*}
 0= - \int_{0}^{1} 2 \langle \partial_s \boldsymbol{\kappa}^{j},
\partial_s\psi^{j} \rangle \,\mathrm{d}s +\int_{0}^{1} (-3 |\boldsymbol{\kappa}^j|^2+\delta)
\left\langle \partial_s\psi^{j},\tau^{j}\right\rangle \,\mathrm{d}s \,,
\end{equation*}
for any $\psi^j \in C^{\infty}_0(0,1)$. That is, reasoning again first for each component 
separately, there exists a constant vector $v$ such that
$$ 2 \partial_s \boldsymbol{\kappa}^{j}- (-3 |\boldsymbol{\kappa}^j|^2+\delta)
\tau^{j} = v \, .$$
A bootstrap argument yields then first that $\gamma^j$ is $H^4$ and then $C^{\infty}$ 
as claimed.

\textit{Boundary terms} 

From~\eqref{secondintegration}
we get that any 
 critical point has to
satisfy
\begin{equation}\label{bcc}
\sum_{i=1}^{M} \left[ 2 \left. \langle \boldsymbol{\kappa}^{i},
 \partial_s \psi^{i}\rangle \right|_0^1  +  \left. 
 \langle -2\partial^\perp_s \boldsymbol{\kappa}^{i}
- |\boldsymbol{\kappa}^i|^2\tau^i+\delta\tau^i, \psi^{i}\rangle
  \right|_0^1 \right] =0\, .
 \end{equation}
Depending on which conditions we impose on the network different 
boundary conditions will be induced.

\textit{A triple junction of three different curves with equal angles}

Let $\Gamma$ be a regular connected network 
only with triple junctions where 
the tangents form equal angles. 
For instance
this is the case 
of Theta--networks. Let $O$ 
be the triple junction 
around which we variate and 
$\gamma^i:[0,1]\to\mathbb{R}^2$, $i\in\{1,2,3\}$,
be the three curves of class $C^\infty$ meeting at $O$. 
Then, 
$$
\gamma^{1}(0)=\gamma^{2}(0)=\gamma^{3}(0)=O\quad\text{and}\quad
\sum_{i=1}^3 \partial_s \gamma^{i}(0) =0\,.
$$
Consider  the 
curves $\psi^{i}:[0,1]\to\mathbb{R}^2$, $i=1,2,3$, in $C^\infty_0([0,1))$
such that
$\psi^{1}(0)=\psi^{2}(0)=\psi^{3}(0)$ and consider the variation network $\tilde{\Gamma}$ 
where the curves $\gamma^i$ are replaced by $\tilde{\gamma}^i:= \gamma^i+t \psi^i$, $i=1,2,3$.
For $|t|$ small this is still a regular network. 

As we want that
at the triple junction $O$
the network
 $\widetilde{\Gamma}$ satisfies also the angular condition
\begin{equation}\label{tangtildazero}
\sum_{i=1}^3\partial_s \tilde{\gamma}^{i}(0)=0\,,
\end{equation}
we have to require something more on the three curves $\psi^{i}$.
Differentiating~\eqref{tangtildazero} it follows that
\begin{equation}\label{diftangzero}
0 = \frac{d}{dt}\left. \sum_{i=1}^3\partial_s \tilde{\gamma}^{i}(0)\right|_{t=0}
= \sum_{i=1}^3\left(\frac{\partial_x\psi^{i}}{\vert\partial_x\gamma^{i} \vert}
-\frac{\partial_x\gamma^{i}}{\vert\partial_x\gamma^{i}\vert^3}
\langle\partial_x\psi^{i},\partial_x\gamma^{i}\rangle \right)(0)
=\sum_{i=1}^3 \partial_s^{\perp} \psi^{i} (0)\,,
\end{equation}
The vectors $\partial_s \psi^{i}$
can be equivalently written as $\partial_s \psi^{i}=\overline{\psi_{s}^{i}}\nu^{i}
+\widetilde{\psi_{s}^{i}}\tau^{i}$, so that $\partial_s^{\perp} \psi^i = \overline{\psi_{s}^{i}}\nu^{i}$. Hence, with this notation, \eqref{diftangzero} becomes 
\begin{equation*}
\sum_{i=1}^3\partial_s^{\perp} \psi^{i} (0) =\sum_{i=1}^3 \overline{\psi_{s}^{i}}(0)\nu^{i}(0)=0\,,
\end{equation*}
that combined with the angular condition 
leads to the following constraint on the variation
\begin{equation}\label{condizioneprimoordine}
 \overline{\psi_{s}^{1}}=\overline{\psi_{s}^{2}}=\overline{\psi_{s}^{3}}
=:\overline{\psi_{s}}\,.
\end{equation}

Now that we have the right variation we can finally derived the
boundary conditions that any critical point in this setting must satisfy at each junction. 
Choosing first variations $\psi^{i} \in C^\infty_0([0,1))$ satisfying $\psi^i(0)=0$, then 
from~\eqref{bcc} it follows that
$$ 0= \sum_{i=1}^{3}    \langle \boldsymbol{\kappa}^{i},
 \partial_s \psi^{i}\rangle(0) = \sum_{i=1}^{3} \langle k^{i} \nu^i,
 \partial_s^{\perp} \psi^{i}\rangle(0) = \sum_{i=1}^{3} \langle k^{i} \nu^i,
 \overline{\psi_{s}} \nu^i\rangle(0)
 =\overline{\psi_{s}}(0) \sum_{i=1}^{3} k^{i} (0)\, , $$
so that  a first boundary condition at the junction is 
$$\sum_{i=1}^{3} k^{i} (0)=0 \, . $$ 
The other term in~\eqref{bcc} leads to another condition. 
Since $\psi^1(0)=\psi^2(0)=\psi^3(0)$ and 
the sum of the tangent vectors is equal to zero at the junction, 
we get
$$ 
\sum_{i=1}^{3} (2 \partial_s k^{i} \nu^i+(k^i)^2 \tau^{i}) (0) =0\,. 
$$

\end{proof}

We derive now the Euler-Lagrange equation for Problem~\eqref{problemdrops}, that is the one satisfied by optimal drops.

\begin{prop}
If a curve $\gamma\in H^2$ is a stationary point for Problem~\eqref{problemdrops},
then it is of class $C^\infty$ and it satisfies
\begin{equation*}
2 (\partial^\perp_s)^2 \boldsymbol{\kappa}+ |\boldsymbol{\kappa}|^2 \boldsymbol{\kappa} - \boldsymbol{\kappa}=0\, \mbox{ on }(0,1),
\end{equation*}
or in terms of the scalar curvature
$$
2\partial_s^2k+k^{3} - k=0\quad\text{for every }\,x\in [0,1]\,,
$$  
together with the boundary conditions
$$
k(0)=k(1)=0\, \mbox{ and } \partial_s k(0)=\pm \partial_s k(1)\,.
$$
\end{prop}
\begin{proof}
Consider $\gamma^i:[0,1]\to\mathbb{R}^2$ 
 a drop of class $H^2$, i.e. in particular satisfying that $\gamma(0)=\gamma(1)$.
Consider a variation of $\gamma$ defined as 
$\widetilde{\gamma}=\gamma+t\psi$
with $t\in\mathbb{R}$, $\psi:[0,1]\to\mathbb{R}^2$ a $C^\infty$ curve
such that  $\psi(0)=\psi(1)$. 
As in the previous proof, one first computes the weak form of the Euler-Lagrange equation 
and then assuming that $\gamma \in H^4$ derives
\begin{align*}
 \frac{d}{dt}F(\widetilde{\gamma})_{|t=0}
 & =  2 \left. \langle \boldsymbol{\kappa},
 \partial_s \psi\rangle \right|_0^1+
 \left. \langle -2\partial_s \boldsymbol{\kappa}+ (-3 |\boldsymbol{\kappa}|^2+1)\tau,
 \psi\rangle \right|_0^1 \\
& +  \int_{0}^{1} 2 \langle (\partial_s^\perp)^2  \boldsymbol{\kappa},
\psi \rangle \,\mathrm{d}s + \int_{0}^{1} ( |\boldsymbol{\kappa}|^2-1)
 \langle \boldsymbol{\kappa},\psi\rangle \,\mathrm{d}s \,,
\end{align*}
and hence the Euler-Lagrange equation. 
Regularity of a critical point can be proven again taking appropriate test-functions 
in the weak form of the Euler-Lagrange equation.

Let us study now which are the
boundary conditions in this case. Choosing first a variation with $\psi(0)
=\psi(1)=0$ we find that any critical point has to satisfy
$$
\boldsymbol{\kappa}(0)=\boldsymbol{\kappa}(1)=0 \mbox{ or equivalently }k(0)=k(1)=0\,.
$$
Then, since $\psi(1)=\psi(0)$, the second condition becomes 
\begin{equation}\label{seccond}
-2\partial_s k(1)\nu(1)+\tau(1)+2\partial_s k(0)\nu(0)-\tau(0)=0\,.
\end{equation}
Without loss of generality we may fix $\tau(1)=(1,0)$. Let us denote 
$\tau(0)=(\tau_1,\tau_2)$. Then, 
it follows that $\nu(1)=(0,1)$, and $\nu(0)=(-\tau_2,\tau_1)$.
Then the vector equality~\eqref{seccond} is equivalent 
to the following system of two scalar equations:
\begin{equation*}
\begin{cases}
1-2\partial_s k(0)\tau_2-\tau_1=0\\
-2\partial_s k(1)+2\partial_s k(0)\tau_1-\tau_2=0\,.
\end{cases}
\end{equation*}
The unique solution of the system is 
$$\tau(0)=\left( \frac{1+4 \partial_s k(0)\partial_s k(1)}{1+4(\partial_s k(0))^2}, \frac{2\partial_s k(0)-2\partial_s k(1)}{1+4(\partial_s k(0))^2}\right) \, .$$
Finally, since $\tau(0)$ has norm equal to one we obtain the second boundary condition 
that a stationary point has to satisfy, namely
\begin{equation*}
\partial_s k(0)=\pm \partial_s k(1)\,.
\end{equation*}
\end{proof}

\section{Scaling properties}\label{secresc}

We want to study 
what happens when we rescale
 a linear combination, 
with positive coefficients, of the elastic energy and the length or the area.
We recall that 
the elastic energy is homogeneous of degree $-1$, the length of degree $1$ 
and the area of degree $2$.
As it turns out
that many computations do not depend on the exact degree of homogeneity 
but only on the fact that there is some type of homogeneity,
 we do the computations here in a general framework.

From now on in this section
let $\alpha,\beta>0$ and let 
$A,B$ be two geometric functional defined on sufficiently regular 
sets $\Gamma$ in $\mathbb{R}^2$, that we do not wish to further specify but that satisfy $B(\Gamma)>0$. We assume that the functional $A$ and $B$ have the following 
scaling properties
\begin{equation}\label{undici}
A(\lambda \Gamma)=\lambda^{-\alpha} A(\Gamma) \quad\text{and}\quad
B(\lambda \Gamma)=\lambda^\beta B(\Gamma) \mbox{ for any }\lambda >0\,.
\end{equation}
In our application 
$\Gamma$ are regular $N$--networks of class $H^2$, 
$A$ is the elastic energy of $\Gamma$, $B$ its total length 
(or if $\Gamma$ is composed by one Jordan curve $
B$ can be also the area enclosed by the curve), 
$\alpha=1$ and $\beta=1$ (or $\beta=2$). Let $\delta>0$, 
and consider the penalized functional 
$F_\delta(\Gamma):=A(\Gamma)+\delta B(\Gamma)$.

\begin{prob}\label{ppenalized}
Consider
\begin{equation}\label{penalized}
\inf\{F_\delta(\Gamma)=A(\Gamma)+\delta B(\Gamma)\vert\; \Gamma \text{ sufficiently regular subset of}\,\mathbb{R}^2\} \, .
\end{equation}
\end{prob}
We call Problem~\ref{ppenalized} the \emph{penalized} minimization problem.

\medskip
\textbf{\textit{Reduction to the case $\delta=1$ by rescaling}}
\medskip

It holds
\begin{equation}\label{stessaforma}
F_{1}(\Gamma)
=\delta^{\frac{-\alpha}{\alpha+\beta}}F_{\delta}\left(\delta^{\frac{-1}{\alpha+\beta}}\Gamma\right)\,.
\end{equation}
Indeed calling  $\widetilde{\Gamma}$ the rescaled network
$\delta^{\frac{-1}{\alpha+\beta}}\Gamma$, by~\eqref{undici}
we obtain
\begin{equation*}
F_{\delta}(\widetilde{\Gamma})=
A(\widetilde{\Gamma})+\delta B(\widetilde{\Gamma})=
\delta^{\frac{\alpha}{\alpha+\beta}}A(\Gamma)
+\delta \left(\delta^{\frac{-\beta}{\alpha+\beta}}B(\Gamma)\right)
=\delta^{\frac{\alpha}{\alpha+\beta}}\left(A(\Gamma)+B(\Gamma)\right)
=\delta^{\frac{\alpha}{\alpha+\beta}} F_{1}(\Gamma)\,.
\end{equation*}
As a consequence of Formula~\eqref{stessaforma}
if  $\Gamma_1$ is a minimizer for $F_{1}$,
then the rescaled
set $\Gamma_2:=\mathrm{R}\Gamma_1$ with $\mathrm{R}=\delta^{\frac{-1}{\alpha+\beta}}$
is a minimizer for  $F_{\delta}$, and vice versa.


Hence, there is no need of considering the more general penalized minimization problem 
given by the functional $F_{\delta}$ with $\delta>0$, we
fix $\delta=1$
and consider the energy 
$F(\Gamma):=F_1 (\Gamma)$.


\medskip

\textbf{\textit{Optimal rescaling}}

\medskip

We have already  used an argument of optimal rescaling when discussing about 
generalization to different angles of the proofs given in~\cite{danovplu}
(see the argument below Figure~\ref{bubble}).
In general it reads as follows.

\begin{lem}[Optimal rescaling]\label{optimalresc}
Consider $\Gamma$ a sufficiently regular subset of $\mathbb{R}^2$,
and $\widetilde{\Gamma}$ the rescaling of $\Gamma$  by the factor
\begin{equation}\label{dodici}
\lambda_{\rm{opt}}:=
\left(\frac{\alpha A(\Gamma)}{\beta B(\Gamma)}\right)^{\frac{1}{\alpha+\beta}}\,.
\end{equation}
Then for every rescaling of factor $\lambda>0$
\begin{equation*}
F(\widetilde{\Gamma})\leq F(\lambda\Gamma)\,.
\end{equation*}
\end{lem}
\begin{proof}
Since by~\eqref{undici}
$F(\lambda\Gamma)=\lambda^{-\alpha} A(\Gamma)+\lambda^\beta B(\Gamma)$, 
by standard computations it is easy to see that the optimal rescaling factor is given by $\lambda=\lambda_{\rm{opt}}$.
\end{proof}

We notice that for each optimal rescaled network, and in particular for  the minimizers of $F$ 
(if there exist), the energy takes the following form
\begin{equation}\label{equie}
F(\lambda_{\rm{opt}}\Gamma)=\left(
\frac{\alpha A(\Gamma)}{\beta B(\Gamma)}\right)^{\frac{-\alpha}{\alpha+\beta}}A(\Gamma)
+\left( \frac{\alpha A(\Gamma)}{\beta B(\Gamma)}\right)^{\frac{\beta}{\alpha+\beta}}B(\Gamma)
=(1+\frac{\alpha}{\beta}) \Big(\frac{\beta}{\alpha}\Big)^{\frac{\alpha}{\alpha+\beta}}A(\Gamma)^{\frac{\beta}{\alpha+\beta}}B(\Gamma)^{\frac{\alpha}{\alpha+\beta}}\,.
\end{equation}
In particular, if $\alpha=\beta$ (as in the case when $A$ is the elastic energy and $B$ the length) there is an equipartition of the energy.

\subsection{Equivalence between the constrained and the penalized problem}\label{secequiv}

Another natural choice in our problem would be to fix the length, instead of penalizing it. 
We speak then of the \textit{constrained minimization problem} 
that in this general framework reads as follows.

\begin{prob}
Consider
\begin{equation}\label{costrp}
\inf\{A(\Gamma)\vert\; \Gamma \text{ sufficiently regular subset of}\;\mathbb{R}^2\,
\text{with}\, B(\Gamma)=B_0\} \,.
\end{equation}
\end{prob}

As a consequence of the rescaling property of the energy, we show now that the penalized and constrained problems are in fact equivalent. 

\begin{lem}\label{probequivalenti}
If $\widetilde{\Gamma}$ is a minimizer for the penalized problem,then 
there exists a rescaling factor $\widetilde{\mathrm{R}}
$ such that
$\widetilde{\mathrm{R}}\widetilde{\Gamma}$
is a minimizer for the constrained problem with the constraint $B(\Gamma)=B_0>0$.
Vice versa if $\overline{\Gamma}$ is a minimizer for the constrained problem,
then there exist a rescaling factor $\overline{\mathrm{R}}$ such that 
$\overline{\mathrm{R}}\,\overline{\Gamma}$ is a minimizer for the penalized problem.
\end{lem}
\begin{proof}

It is not restrictive to suppose that 
$\widetilde{\Gamma}$ is a minimizer for the penalized problem
with $\delta=1$ (by Lemma~\ref{stessaforma}), that is
$$
F(\widetilde{\Gamma})=A(\widetilde{\Gamma})+B(\widetilde{\Gamma})
\leq A(\Gamma)+B(\Gamma)
$$
for all admissible sets $\Gamma$. If 
$B(\widetilde{\Gamma})=B_0$, then $\widetilde{\Gamma}$
is clearly  a minimizer for the constrained problem with constraint $B(\Gamma)=B_0$. Indeed
let $\Gamma$ such that $B(\Gamma)=B_0$ then
$$
A(\widetilde{\Gamma})
=A(\widetilde{\Gamma})+B(\widetilde{\Gamma})-B_0
\leq A(\Gamma)+B(\Gamma)-B_0=A(\Gamma)\,.
$$
Otherwise, $B(\widetilde{\Gamma})=\rho B_0$ for some $\rho \ne 1$.  
We claim that $\rho^{-\frac{1}{\beta}} \widetilde{\Gamma}$ is a minimizer of the constrained problem. Indeed let $\Gamma$ such that $B(\Gamma)=B_0$ then since $B(\rho^{\frac{1}{\beta}}\Gamma)= \rho B_0$ with the argument above it follows
$$ A(\rho^{-\frac{1}{\beta}} \tilde{\Gamma})= \rho^{\frac{\alpha}{\beta}} A(\tilde{\Gamma}) \leq \rho^{\frac{\alpha}{\beta}} A(\rho^{\frac{1}{\beta}}\Gamma)= A(\Gamma) \,, $$
that is what we claimed.

Suppose instead that $\overline{\Gamma}$ is a minimizer for the constrained problem
with $B(\Gamma)=B_0$. We want to show that $\overline{\mathrm{R}}\,\overline{\Gamma}$ 
with $\overline{\mathrm{R}}$ given by the optimal rescaling~\eqref{dodici} is
a minimizer for the penalized problem.
Let $\Gamma$ be a competitor for $F$ with $B(\Gamma)=C_0$, then there exist $\lambda>0$
such that $B(\lambda\overline{\Gamma})=C_0$. By the same computations done above $A(\lambda\overline{\Gamma})\leq A(\Gamma)$. 
Using Lemma~\ref{optimalresc} we get
$$
A(\overline{\mathrm{R}}\,\overline{\Gamma})+B(\overline{\mathrm{R}}\,\overline{\Gamma})
\leq A(\lambda\overline{\Gamma})+B(\lambda\overline{\Gamma})=
A(\lambda\overline{\Gamma})+C_0\leq A(\Gamma) +C_0 = 
A(\Gamma)+B(\Gamma)\,,
$$
for all admissible $\Gamma$.
\end{proof}

\section{The energy of the \lq\lq Figure Eight\rq\rq}\label{bulgari}

The proof in \cite{danovplu} that an optimal Theta--network is not degenerate 
is based on a good approximation of the energy of the optimal rescaling of the
 \lq\lq Figure Eight\rq\rq~$\mathcal{F}$, 
 the unique (up to multiple coverings) closed elastica with a self-intersection, 
 as we have discussed already in Subsection~\ref{theta}. 
In this section we give the arguments that lead us to the value $F(\mathcal{F})\approx 21.2075$.

Langer and Singer already observed in \cite{langsing} that the Euler-Lagrange equation
 of the elasticae 
$$2 \partial_{ss} k+\left(k\right)^{3} -\delta k=0 \mbox{ on }(0,1)\, ,$$
(see \eqref{eulero}) can be integrated using Jacobi-Elliptic functions. 
In~\cite{explicitelasticae} the authors find a dynamical system 
that the components $(x,y)$ of an elastica parametrized by arc-length satisfy. 
From this dynamical system description one is able to find an explicit parametrization 
of the \lq\lq Figure Eight\rq\rq ~depending only on well defined parameters. 
Thanks to this representation we are able to compute in this section the energy of $\mathcal{F}$.

\subsection{From the equation of the elasticae to a dynamical system}
For completeness we give here the details of the derivation 
of the dynamical systems for the elasticae 
presented  in~\cite{explicitelasticae}.

\begin{prop}\label{daeqasist}
Consider $\delta\in\mathbb{R}$, $\delta>0$ fixed and $I$ an interval.
Let $\gamma:I\to \mathbb{R}^2$, 
$\gamma= (x,y)^t 
$
be a smooth regular curve parametrized by arc-length and let $k$ denote its scalar curvature. 
If $\gamma$ is an elastica, that is a solution of 
\begin{equation}\label{zero}
2 \partial_s^2 k+k^{3} -\delta k=0\,,
\end{equation}
then up to isometries of $\mathbb{R}^2$,
for some $\lambda\in\mathbb{R}$,
$\lambda> 0$
 the couple of its coordinates $(x,y)^t$ satisfies
\begin{equation}\label{sistema}
\begin{cases}
x''=\lambda yy'\\
y''=-\lambda yx'\,.
\end{cases}
\end{equation}
\end{prop}
\begin{proof}
Multiplying~\eqref{zero} by $\partial_s k$ and integrating we find that there exists
a constant $b\in\mathbb{R}$ such that
\begin{equation*}
(\partial_s k)^2+\frac14 k^4-\frac\delta2k^2=b\,,
\end{equation*} 
or equivalently
\begin{equation}\label{uno}
(\partial_s k)^2+\left(\frac{k^2}{2} -\frac\delta2\right)^2= b+\frac{\delta^2}{4}\,.
\end{equation}
Now let consider $\lambda^2=b+\frac{\delta^2}{4}$ and 
the regular, smooth curve $\tilde{\gamma}:I\to\mathbb{R}^2$,
$\tilde{\gamma}= (x,y)^t$ 
with
\begin{align}
y(s)&=-\frac{k(s)}{\lambda}\label{due}\\
x'(s)&=\frac{1}{\lambda}\left(\frac{\lambda^2}{2}y^2(s)-\frac\delta2\right)
=\frac{1}{\lambda}\left(\frac{k^2(s)}{2}-\frac\delta2\right)\,.\label{tre}
\end{align}
By~\eqref{uno} it follows that
\begin{equation*}
\left(x'\right)^2+\left(y'\right)^2=
\frac{1}{\lambda^2}\left(\frac{k^2}{2}-\frac\delta2\right)^2+
\frac{\left(\partial_s k\right)^2}{\lambda^2}
=\frac{1}{\lambda^2}\left(b+\frac{\delta^2}{4}\right)=1\,,
\end{equation*}
that is $\tilde{\gamma}$ is parametrized by arclength. Moreover, the components of $\tilde{\gamma}$ satisfy 
the ODE system~\eqref{sistema}: indeed the first equation
follows differentiating~\eqref{tre}, while for the other using~\eqref{due}, 
\eqref{zero} and~\eqref{tre} we find
\begin{align*}
y''&=\left(-\frac{\partial_s k}{\lambda}\right)' = -\frac{\partial_s^2 k}{\lambda}=-\frac1\lambda\left(-\frac12 k^3+\frac\delta2 k\right)\\
&=\frac{k}{\lambda}	\left(\frac12 k^2-\frac\delta2\right)=kx'=-\lambda yx'\,.
\end{align*}
Now we compute the scalar curvature $\tilde{k}$ 
of the curve $\tilde{\gamma}$. Since $\tilde{\gamma}$ is parametrized by arclength
we have $\tilde{k}=x'y''-x''y'$, then 
\begin{align*}
\tilde{k}&=-\lambda y(x')^2-\lambda y(y')^2=-\lambda y=k\,,
\end{align*}
so, up to isometries of $\mathbb{R}^2$, we obtained that the curve $\tilde{\gamma}$
is in fact the curve $\gamma$.
\end{proof}

\begin{lem}\label{dasistaeq}
Let $I$ be an interval and $\gamma:I\to \mathbb{R}^2$, 
$\gamma=(x,y)^t
$
be a smooth regular curve 
parametrized by arclength,
such that  $(x,y)$ satisfies \eqref{sistema} 
for some $\lambda>0$.
Then the curvature of $\gamma$ is given by $k=-\lambda y$
and $\gamma$ solves 
$2k_{ss}+k^{3} -\delta k=0$ for $\delta=-2\lambda\mu$ for some $\mu\in\mathbb{R}$.
\end{lem}
\begin{proof}
By definition of curvature and using the fact that the components of $\gamma$ 
solve the system~\eqref{sistema} we get
\begin{equation*}
k=x'y''-x''y'=-\lambda y \left(x'\right)^2-\lambda y\left(y'\right)^2=-\lambda y\,.
\end{equation*}
Moreover integrating the first equation of the system~\eqref{sistema}
it follows $x'=\frac{\lambda}{2}y^2+\mu$ for some $\mu\in\mathbb{R}$,
that inserted in the second equation gives
\begin{align*}
2 \partial_s^2 k=-2\lambda y''=2\lambda^2 yx'=2\lambda^2y\left(\frac\lambda2 y^2+\mu\right)=
\lambda^3 y^3+2\lambda^2\mu y=-k^3-2\lambda\mu k\,.
\end{align*}
Hence choosing $\mu=-\frac{\delta}{2\lambda}$ the claim follows.
\end{proof}

\subsection{Parametrization and energy of the   \lq\lq Figure Eight\rq\rq}

Consider a regular smooth curve $\gamma=(x,y)^t$ parametrized by arclength and
 with coordinates satisfying  the system~\eqref{sistema} for some $\lambda>0$. Then,
\begin{align}\label{propxprimo}
&x'=\frac\lambda2 y^2+\mu\quad\text{for some}\;\mu\in\mathbb{R}\,,\\
\label{quattro}
\text{and }&y''+\frac{1}{2}\lambda^2 y^3+\lambda\mu y=0\,.
\end{align}
Multiplying~\eqref{quattro} by $y'$ and integrating we get
\begin{equation*}
\left(y' \right)^2+\frac{\lambda^2}{4}y^4+\lambda\mu y^2 =\left(y' \right)^2+ \left(x' \right)^2-\mu^2 =c
\end{equation*}
with $c\in\mathbb{R}$.
Since $\gamma$ is parametrized by arclength $c=1-\mu^2$.
Hence 
\begin{equation}\label{venti}
\left(y' \right)^2=\frac{\lambda^2}{4}\left(\frac{2(1-\mu)}{\lambda}-y^2\right)
\left(y^2+\frac{2(1+\mu)}{\lambda}\right)\,,
\end{equation}
and we see that necessarily $\mu<1$.

The equation~\eqref{venti} can be integrated
as it has already been done by Langer and Singer~\cite{langsing1, langsing}.
In the following we briefly recall
the definition of  Jacobi--Elliptic functions
(see~\cite{abra})
 needed to integrate~\eqref{venti}.

\begin{defn}\label{jef}
For $m \in [0,1]$ and $\varphi\in \mathbb{R}$ let
$$
u=\int_0^\varphi \frac{d\theta}{\left(1-m\sin^2\theta\right)^\frac{1}{2}}\,.
$$
One considers then the inverse of this function. 
The angle $\varphi:={\rm{am}}(u,m)$ is called \emph{amplitude}.
Moreover ${\rm{cn}}(u,m):=\cos\varphi$.
With a little abuse of notation we define the \emph{elliptic integral of the second kind} as
$$
E(\varphi, m)=E(u,m)=\int_0^\varphi \left(1-m\sin^2\theta\right)^\frac12\,d\theta
=(1-m)u+m\int_{0}^u \rm{cn}^2(w,m)\,dw\,,
$$
where the relation between $u$ and $\varphi$ is the same as before.
Moreover we introduce the real quarter period $K(m)$ defined by
$$
K(m)=\int_0^\frac\pi2 \frac{d\theta}{\left(1-m\sin^2\theta\right)^\frac{1}{2}}\,.
$$
\end{defn}

Now we are ready to integrate \eqref{venti}.
Depending on the value of the parameter $\mu$ in~\eqref{venti} 
we get different solutions. 
As we are considering only the case $\delta=1$, 
by Lemma~\ref{dasistaeq}, we are interested only on negative $\mu$
and on values of $\lambda$ such that $-2 \lambda \mu=1$.
For $\mu \in (-1,0)$~\eqref{venti} can be integrated. One finds 
\begin{equation}\label{y}
y(s)=a\, {\rm{cn}}(\sqrt{\lambda}s,m)
\end{equation}
with $a^2=\frac{2(1-\mu)}{\lambda}$,  $m=\frac{1-\mu}{2}$.
Notice that here we fix the axes in such a way that $y(0)$ is the maximal value of $y$.
Using~\cite[pag.  573--574]{abra} one sees that $y$ as defined in~\eqref{y}
satisfies~\eqref{venti}.

Integrating~\eqref{propxprimo} with  $y$ as defined in~\eqref{y}, one obtains
\begin{equation}\label{x}
x(s)=\frac{2}{\sqrt{\lambda}} E({\rm{am}} (\sqrt{\lambda}s,m),m)-s\,.
\end{equation} 
Using the properties of $E$~\cite[pag.  589]{abra}
it is easy to see $x(0)=0$, $x'(0)=1$, $x'(s)$ satisfies~\eqref{propxprimo}
and $x$ is an odd function.

A proper rescaled \lq\lq Figure Eight\rq\rq ~
satisfies the equation of the elasticae for $\delta=1$
and hence, from Proposition~\ref{daeqasist} 
we know that its coordinates solve the system~\eqref{sistema} for some $\lambda$ (and hence some $\mu$). Since by \cite{langsing1, langsing, explicitelasticae}~$\mathcal{F}$ is characterised as the only closed elastica with at least a self-intersection (up to multiple coverings), if we show that for a value of $\mu \in (-1,0)$ the curve $(x(s),y(s))^t$ (with $x$ and $y$ as in \eqref{x}, \eqref{y}) is closed and has a self-intersection, we are sure that this is a parametrization of~$\mathcal{F}$.  

By properties of the Jacobi elliptic functions, $y$ defined in \eqref{y} is even, periodic with minimal period $4K(m)/\sqrt{\lambda}$, monotonically decreasing for $s \in (0,2K(m)/\sqrt{\lambda})$ and
$$
y\left(\frac{K(m)}{\sqrt{\lambda}}\right)=0 
\mbox{ and } y\left(2\frac{K(m)}{\sqrt{\lambda}}\right)= - y(0) \, .
$$ 
Now we argue that for a unique value of $m$ and hence of $\mu$, $x$ defined in \eqref{x} is periodic with period $2K(m)/\sqrt{\lambda}$ with $x(0)=0=x(K(m)/\sqrt{\lambda})$. 
Indeed, one first observes that
$$ {\rm{am}} (\sqrt{\lambda}s + 2 K(m),m) = {\rm{am}} (\sqrt{\lambda}s ,m)+ \pi \, ,$$
and
$$E(\varphi +\pi,m) = \int_0^{\pi} (1-m (\sin\theta)^2)^{\frac12} d \theta + E(\varphi ,m) = 2 E(m) + E(\varphi ,m) ,$$
with 
$$E(m)= \int_0^{\frac{\pi}{2}} (1-m (\sin\theta)^2)^{\frac12} d \theta \, . $$
If follows that
\begin{align*}
x(s+2K(m)/\sqrt{\lambda}) & =\frac{2}{\sqrt{\lambda}} E({\rm{am}} (\sqrt{\lambda}s + 2 K(m),m),m)-s-2\frac{K(m)}{\sqrt{\lambda}}\\
& = \frac{2}{\sqrt{\lambda}} E({\rm{am}} (\sqrt{\lambda}s ,m)+ \pi,m)-s-2\frac{K(m)}{\sqrt{\lambda}} \\
& = \frac{2}{\sqrt{\lambda}} E({\rm{am}} (\sqrt{\lambda}s ,m),m) +\frac{4}{\sqrt{\lambda}} E(m) -s-2\frac{K(m)}{\sqrt{\lambda}}\\
& = x(s) +\frac{2}{\sqrt{\lambda}} (2 E(m)-K(m)) \, .
\end{align*}
Since $E$ is decreasing in $m$ and $K$ is increasing in $m$ 
we see that there is exactly a value $\overline{m}$ of $m$ between $0$ and $1$ 
such that the claimed periodicity follows. 
By numerical approximation
we obtain that the value is $\overline{m}\approx0,826115$, 
and consequently we get the value of $\mu$ from the relation $\overline{m}=\frac{1-\mu}{2}$.
As we are looking at $F(\mathcal{F})$, $\delta=1$, we get the value of 
$\lambda$ from the relation $1=-2\lambda\mu$ expressed in Lemma~\ref{dasistaeq}.

          \begin{figure}[H]
\begin{center}
\begin{tikzpicture}
\draw[black!50!white, shift={(1.5,0)}, scale=1, rotate=-90]
(0,0)to[out= -45,in=135, looseness=1] (0.1,-0.1)
(0,0)to[out= -135,in=45, looseness=1] (-0.1,-0.1)
(0,-0.3)to[out= 90,in=-90, looseness=1] (0,0.003);
\draw[black!50!white, shift={(0,2.5)}, scale=1, rotate=0]
(0,0)to[out= -45,in=135, looseness=1] (0.1,-0.1)
(0,0)to[out= -135,in=45, looseness=1] (-0.1,-0.1)
(0,-0.3)to[out= 90,in=-90, looseness=1] (0,0.003);
\draw[black!50!white]
(-0.75,0)--(1.5,0)
(0,-0.75)--(0,2.5);
\draw
(0,1.73)to[out= 0,in=90, looseness=1] (0.56,1.1)
(0.56,1.1)to[out= -90,in=50, looseness=1] (0,0);
\path[font=\footnotesize, shift={(0,0)}] 
(0,1.9)[right] node{$\left(x(0),y(0)\right)$}
(0,-0.3)[right] node{$\left(x(\overline{s}),y(\overline{s})\right)$};
 \end{tikzpicture}
\end{center}  

\caption{A piece of the \lq\lq Figure Eight\rq\rq.}\label{quartootto}

          \end{figure}
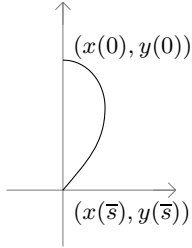
					
We can finally compute the energy of the \lq\lq Figure Eight\rq\rq.
					
\begin{prop}
Consider the optimal rescaling $\mathcal{F}$ of the \lq\lq Figure Eight\rq\rq. Then, $F(\mathcal{F})\approx 21.2075$.
\end{prop}
\begin{proof}
Let $\overline{s}:=K(\overline{m})/\sqrt{\lambda}$. We already know that
$x(0)=x(\overline{s})=0$, $y(0)=y_{\max}$, $y(\overline{s})=0$
and considering values of $s$ between $0$ and $\overline{s}$ 
we parametrize a quarter of the \lq\lq Figure Eight\rq\rq~see Figure \ref{quartootto}.
Again from Lemma~\ref{dasistaeq} (cfr.~\cite{explicitelasticae}) we know 
that $k(s)=-\lambda y(s)$,
hence 
$$
E(\mathcal{F})=\int_{\mathcal{F}}k^2\,\mathrm{d}s=4\int_0^{\overline{s}}k^2\,\mathrm{d}s=
4\int_0^{\overline{s}}\lambda^2a^2cn^2(\sqrt{\lambda}s,\overline{m})\,\mathrm{d}s
$$
and
$$
L(\mathcal{F})=4\overline{s}\,.
$$
By computer assisted computations
we get 
$E(\mathcal{F})\approx21.2075$.
\end{proof}


%
%
%
%
%
\end{document}